\documentclass{elsarticle}
\journal{Stochastic Processes and their Applications}

\usepackage{mathrsfs} 
\usepackage{amssymb} 
\usepackage{amsmath}
\usepackage{amsthm} 
\usepackage{dsfont} 
\usepackage{mathtools} 
\usepackage{ascmac} 
\usepackage{color}
\usepackage{bm} 
\usepackage{comment}
\usepackage[all]{xy} 

\newcommand{\R}{\mathbb{R}}
\newcommand{\N}{\mathbb{N}}

\newcommand{\F}{\mathcal{F}}
\newcommand{\D}{\mathcal{D}}
\newcommand{\B}{\mathcal{B}}

\numberwithin{equation}{section} 

\newtheorem {rem}{Remark}[section]
\newtheorem {cor}[rem]{Corollary}
\newtheorem {thm}[rem]{Theorem}
\newtheorem {lem}[rem]{Lemma}

\newtheorem {prop}[rem]{Proposition}

\newtheorem {assu}[rem]{Assumption}

\usepackage[top=2cm, bottom=2cm, left=2cm, right=2cm]{geometry}

\begin{document}
	
\begin{frontmatter}
	\title{Weak Convergence of Approximate reflection coupling and its Application to Non-convex Optimization}
	\author{Keisuke Suzuki}
	\address{Biometrics Research Laboratories, NEC Corporation, 1753, Shimonumabe, Nakahara-Ku,	Kawasaki, Kanagawa 211-8666, Japan}
	\ead{keisuke.suzuki.334@nec.com}
	\begin{abstract}
		In this paper, we propose a weak approximation of the reflection coupling (RC) for stochastic differential equations (SDEs), and prove it converges weakly to the desired coupling. 
		In contrast to the RC, the proposed approximate reflection coupling (ARC) need not take the hitting time of processes to the diagonal set into consideration and can be defined as the solution of some SDEs on the whole time interval.
		Therefore, ARC can work effectively against SDEs with different drift terms. 
		As an application of ARC, an evaluation on the effectiveness of the stochastic gradient descent in a non-convex setting is also described. 
		For the sample size $n$, the step size $\eta$, and the batch size $B$, we derive uniform evaluations on the time with orders $n^{-1}$, $\eta^{1/2}$, and $\sqrt{(n - B) / B (n - 1)}$, respectively.
	\end{abstract}
	\begin{keyword}
		Reflection Coupling; Stochastic Differential Equation; Gradient Descent; Non-convex Optimization
	\end{keyword}
\end{frontmatter}

\section{Introduction}

Finding a good coupling $\gamma$ between two probability measures $\mu$ and $\nu$ is important for evaluating the difference between them. 
Here, $\gamma$ is said to be a coupling between $\mu$ and $\nu$ if each marginal distributions of $\gamma$ coincide with $\mu$ and $\nu$, respectively. 
In fact, the Wasserstein distance \citep{Villani}, which measures the difference between two probability measures through good couplings of them, is bounded by the Kullback--Leibler divergence \citep{Bolley} and is one direction to connect the probability theory with the information theory. 
In particular, it is worth finding a good coupling between laws of solutions of stochastic differential equations (SDEs), which appear frequently in applications. 

For a good coupling between laws of solutions of SDEs, \citep{Torg} introduced the reflection coupling (RC). 
For a continuously differentiable function $H : \R^d \to \R$ and a $d$-dimensional Brownian motion $W$, for example, we consider Langevin dynamics along with the gradient $\nabla H$ of $H$. 
\begin{align}
	\label{LSDE_along_H}
	dX_t = - \nabla H(X_t) dt + dW_t.
\end{align}
Then, the RC for (\ref{LSDE_along_H}) is defined by
\begin{align}
	\label{Reflection_Coupling_along_H}
	dY_t 
	= - \nabla H(Y_t) dt + (I_d - 2 e_t e_t^\top) d W_t,\quad t < T,\quad Y_t = X_t,\quad t \geq T. 
\end{align}
Here, $I_d$ denotes the $d \times d$ identity matrix, $T = \inf \{ t \geq 0 \mid X_t = Y_t \}$ is the hitting time of $(X, Y)$ to the diagonal set, and $e_t = (X_t - Y_t) / \| X_t - Y_t \|_{\R^d}$. 
Thus, for each $t < T$, the orthogonal matrix $I_d - 2 e_t e_t^\top$ defines a plane symmetric transformation with respect to a plane orthogonal to $e_t$. 
Therefore, denoting the indicator function of a set $A$ by $\chi_A$, $W_t^\prime = \int_0^t (I_d - 2 \chi_{\{ s< T \}} e_s e_s^\top) dW_s$ defines the Brownian motion whose instant increments are plane symmetric with those of $W$. 
In particular, $Y$ is also a weak solution of (\ref{LSDE_along_H}) by the Markov property of $X$, and $(X, Y)$ is a process that approaches to the diagonal set by the symmetry of $W$ and $W^\prime$.
In fact, \citep{Ebe} proved the inequality $E[\rho_2(X_t, Y_t)] \leq e^{-c t} E[\rho_2(X_0, Y_0)]$ for some $c > 0$ and a function $\rho_2 : \R^d \times \R^d \to \R$ that satisfies $\| x - y \|_{\R^d} \leq C \rho_2(x, y)$ for some $C > 0$. 
Thus, the $1$-Wasserstein distance between Langevin dynamics (\ref{LSDE_along_H}) with different initial values converges to $0$ as $t$ tends to infinity. 

However, the RC does not work for Langevin dynamics with different or functional drift terms since, in this case, $Y$ is not the weak solution of the SDE it should solve because of the definition of it for $t \geq T$.  
For Langevin dynamics with different drift terms, \citep{Ebe2} proposed the sticky coupling as a substitute for the RC. 
However, the sticky coupling can evaluate only the probability that this coupling is out of the diagonal set for each fixed time. 
That is, we cannot evaluate important quantities like $1$-Wasserstein distance by the sticky coupling. 
Therefore, we can conclude that good couplings for Langevin dynamics that are more general than (\ref{LSDE_along_H}) have not been found. 

In this paper, we propose the following approximate reflection coupling (ARC) that also works for Langevin dynamics that are more general than (\ref{LSDE_along_H}). 
\begin{align}
	\label{Approximate_Reflection_Coupling_along_H}
	dY_t^{(\varepsilon)} 
	= - \nabla H(Y_t^{(\varepsilon)}) dt + (I_d - 2 h_\varepsilon(\| X_t - Y_t^{(\varepsilon)} \|_{\R^d}) e_t^{(\varepsilon)} {e_t^{(\varepsilon)}}^\top) d W_t,\quad t \geq 0.
\end{align}
Here, $e_t^{(\varepsilon)} = (X_t - Y_t^{(\varepsilon)}) / \| X_t - Y_t^{(\varepsilon)} \|_{\R^d}$ and $h_\varepsilon : \R \to [0, 1]$ is an arbitrary $C^1$-function that values $0$ in a neighborhood of the origin and $1$ outside of another neighborhood. 
(\ref{Approximate_Reflection_Coupling_along_H}) has an advantage in that it can define $Y^{(\varepsilon)}$ as the solution of the SDE defined on the whole time interval. 
In other words, in contrast to the RC, (\ref{Approximate_Reflection_Coupling_along_H}) does not need the particular definition of $Y^{(\varepsilon)}$ for $t$ after the hitting time to the diagonal set and can handle the case of different or functional drift terms. 
The first main result, Theorem \ref{Main_Theorem_Weak_Conv}, states that the ARC defined between a semi-martingale and Langevin dynamics, which is a more general case than (\ref{LSDE_along_H}) and (\ref{Approximate_Reflection_Coupling_along_H}), converges weakly to the desired coupling. 

The second main result is an application of Theorem \ref{Main_Theorem_Weak_Conv} to the theoretical analysis of stochastic gradient Langevin dynamics (SGLD) in a non-convex setting. 
As we will see later, the problem to evaluate the effectiveness of SGLD is equivalent to the problem to evaluate the difference between Langevin dynamics with different drift terms. 
Thus, we can derive a sharp evaluation of the effectiveness of SGLD by using ARC. 

This paper is organized as follows. 
In Section \ref{SEC_Main_Result}, we give accurate statements of our main results of weak convergence of ARC and evaluations for SGLD. 
Section \ref{SEC_Main1_Proof} describes the proof of the first main result, Theorem \ref{Main_Theorem_Weak_Conv}, and Sections \ref{SEC_First_Appli}, \ref{SEC_Second_Appli}, and \ref{SEC_Third_Appli} are devoted to the proofs of the three inequalities given in the second main result, Theorem \ref{Main_Theorem_SGLD}. 
Finally, auxiliary results are given in Appendix. 

\section{Main Result}
\label{SEC_Main_Result}

To formulate our first main result, we introduce the following notations. 
$\{ \F_t \}$ is a filtration that satisfies usual condition (Definition 1.2.25 in \citep{kara}) and $W$ is a $d$-dimensional $\{ \F_t \}$-Brownian motion. 
$X_0$ and $Y_0$ are $\F_0$-measurable $\R^d$-valued random variables and $V$ is a $d$-dimensional $\{ \F_t \}$-adapted continuous process with bounded variation and initial value $V_0 = 0$. 
$C([0, \infty); \R^d)$ denotes the set of all continuous functions from $[0, \infty)$ to $\R^d$ and a functional $G : [0, \infty) \times C([0, \infty); \R^d) \to \R^d$ is progressively measurable (Definition 3.5.15, \citep{kara}). 
Then, we impose the following assumption. 
Here, $L^p(\Omega; \R^d)$ is the set of all $p$-th integrable random variables from $\Omega$ to $\R^d$ and $\check{V}$ is the total variation of $V$. 

\begin{assu}
	\label{Assum_Main1}
	For some $p > 2$, $X_0, Y_0 \in L^p(\Omega; \R^d)$ and $E[\|\check{V}_t\|_{\R^d}^p] < \infty$ holds for all $t \geq 0$. 
	In addition,  there exists a constant $K(t)$ for all $t \geq 0$ such that the following inequalities hold for all $\varphi, \psi \in C([0, \infty); \R^d)$. 
	\begin{align}
		&\| G(s, \varphi) - G(s, \psi) \|_{\R^d} 
		\leq K(t) \sup_{0 \leq u \leq s} \| \varphi(u) - \psi(u) \|_{\R^d},\quad 0 \leq s \leq t, \label{Lipschitz_Functional} \\
		&\| G(s, \varphi) \|_{\R^d} 
		\leq K(t) \left( 1 + \sup_{0 \leq u \leq s} \| \varphi(u) \|_{\R^d} \right),\quad 0 \leq s \leq t. \label{Linear_Growth_Functional}
	\end{align}
\end{assu}

\noindent Under Assumption \ref{Assum_Main1}, for a constant $\sigma \neq 0$, we define the semi-martingale $X$ as 
\begin{align}
	\label{Semimartingale}
	X_t
	= X_0 + V_t + \sigma W_t,\quad t \geq 0
\end{align}
and denote the solution of the functional SDE
\begin{align}
	\label{Functional_SDE}
	Y_t 
	= Y_0 + \int_0^t G(s, Y) ds + \sigma W_t,\quad t \geq 0
\end{align}
by $Y$. 
Thus, Gronwall's lemma yields $E[\sup_{0 \leq s \leq t} \| Y_s \|_{\R^d}^p] < \infty$ for all $t \geq 0$. 

Under the aforementioned notations, for all $\varepsilon > 0$, we define the ARC between $X$ and $Y$ by
\begin{align}
	\label{General_Approximate_Reflection_Coupling}
	\begin{cases}
		dY_t^{(\varepsilon)} 
		= G(t, Y^{(\varepsilon)}) dt + \sigma (I_d - 2 h_\varepsilon( \| X_t - Y_t^{(\varepsilon)} \|_{\R^d} ) e_t^{(\varepsilon)} {e_t^{(\varepsilon)}}^\top) d W_t, \\
		Y_0^{(\varepsilon)} = Y_0. 
	\end{cases}
\end{align}
Here, $e_t^{(\varepsilon)} = (X_t - Y_t^{(\varepsilon)}) / \| X_t - Y_t^{(\varepsilon)} \|_{\R^d}$ and $h_\varepsilon : \R \to [0, 1]$ is an arbitrary $C^1$-function that satisfies
\begin{align}
	\label{CouplingDevice}
	\begin{cases}
		h_\varepsilon(a) = 0, & | a | \leq \varepsilon, \\
		h_\varepsilon(a) = 1, & | a | \geq 2 \varepsilon.
	\end{cases}
\end{align}
Our first main result is stated as follows, where $\mathcal{L}(Z)$ denotes the law of a random variable $Z$. 

\begin{thm}
	\label{Main_Theorem_Weak_Conv}
	Under Assumption \ref{Assum_Main1}, we can take a subsequence $\varepsilon_\ell$ and a coupling $\gamma$ between $\mathcal{L}(X)$ and $\mathcal{L}(Y)$ so that $\mathcal{L}(X, Y^{(\varepsilon_\ell)})$ converges weakly to $\gamma$. 
\end{thm}

Next, to formulate our second main result, we introduce the following notations. 
$\mathcal{Z}$ is the set of all data points and $\ell(w; z)$ denotes the loss on $z \in \mathcal{Z}$ for a parameter $w \in \R^d$. 
$z_1, \dots, z_n$ are independent and identically distributed (IID) samples generated from the distribution $\D$ on $\mathcal{Z}$. 
For the batch size $B \leq n$, $\{ I_k \}_{k=1}^\infty$ denotes the sequence of random extraction from $\{ 1, \dots, n \}$ with size $B$. 
Finally, for each parameter $w \in \R^d$, we define the expected loss, the empirical loss, and its mini-batch by $L(w) = E_{z \sim \D}[\ell(w; z)]$, $L_n(w) = \frac{1}{n} \sum_{i=1}^n \ell(w; z_i)$ and $L_{n, k}(w) = \frac{1}{B} \sum_{i \in I_k} \ell(w; z_i)$, respectively. 
For the step size $\eta > 0$ and the inverse temperature $\beta > 0$, we define the SGLD $X^{(n, \eta, B)}$ as 
\begin{align}
	\label{SGLD_along_Ln_Batch}
	X_t^{(n, \eta, B)}
	= X_{k \eta}^{(n, \eta, B)} - (t- k \eta) \nabla L_{n, k}(X_{k \eta}^{(n, \eta, B)}) + \sqrt{2 / \beta} (W_t - W_{k \eta}), \qquad k \eta \leq t <(k+1) \eta. 
\end{align}

There are many existing works \citep{Bertsekas, Kumar, Moulines, Abbasi, Mackey2, Ge, Ge2, Qian, Majka, Kavis, Mou, Taiji, Ragi, Taiji2, Xu, Zhang, Liang} aimed at evaluating the effectiveness of SGLD. 
In almost all of them, the goal is to derive a sharp bound for fixed large $t$ to the quantity
\begin{align}
	\label{SGLD_Goal}
	E[L(X_t^{(n, \eta, B)})] - \min_{w \in \R^d} L(w) 
\end{align}
in terms of $n$, $\eta$, $B$ and $\beta$. 

Let two processes $X^{(n)}$ and $X^{(n, \eta)}$ defined by
\begin{align}
	&dX_t^{(n)} = -\nabla L_n(X_t^{(n)}) dt + \sqrt{2 / \beta} dW_t,\qquad t \geq 0, \label{LSDE_along_Ln} \\
	&X_t^{(n, \eta)} = - (t - k\eta) \nabla L_n(X_{k \eta}^{(n, \eta)}) + \sqrt{2 / \beta}  (W_t - W_{k \eta}),\qquad k \eta \leq t < (k+1) \eta \label{SGLD_along_Ln}
\end{align}
have the same initial values as $X^{(n, \eta, B)}$. 
Then, the quantity (\ref{SGLD_Goal}) can be decomposed as 
\begin{align}
	\label{SGLD_Err_Decompose}
	&E[L(X_t^{(n, \eta, B)})] - \min_{w \in \R^d} L(w) 
	= \{E[L(X_t^{(n, \eta, B)})] - E[L_n(X_t^{(n)})]\}
	+ \{ E[L_n(X_t^{(n)})] - \min_{w \in \R^d} L(w) \}.
\end{align}
According to (3.26) in \citep{Ragi} and the result in \citep{Ebe}, the second term in the R.H.S of (\ref{SGLD_Err_Decompose}) is bounded by the form of constant times $e^{- c t} + d \beta^{-1} \log(\beta / d + 1)$.
Thus, the problem to derive a bound to (\ref{SGLD_Goal}) can be reduced to the problem to evaluate the first term in the R.H.S of (\ref{SGLD_Err_Decompose}), which is the difference between Langevin dynamics with different drift terms, and ARC can be applied to evaluate it. 

To evaluate the first term in the R.H.S of (\ref{SGLD_Err_Decompose}), we impose the following assumption, which is commonly used in previous works.
Here, $C^k(\R^d; \R)$ is the set of all $C^k$-functions from $\R^d$ to $\R$. 

\begin{assu}
	\label{Assum_Main2}
	The same initial value of (\ref{SGLD_along_Ln_Batch}), (\ref{LSDE_along_Ln}) and (\ref{SGLD_along_Ln}) belongs to $L^4(\Omega; \R^d)$. 
	The loss $\ell(w; z)$ is nonnegative and satisfies $\sup_{z \in \mathcal{Z}} |\ell(0; z)| < \infty$ and $\sup_{z \in \mathcal{Z}} \|\nabla \ell(0; z)\|_{\R^d} \leq A$ for some $A > 0$.
	Thus, the expected loss $L(w) = E_{z \sim \D}[\ell(w; z)]$ is well-defined. 
	In addition, $\ell(\cdot; z) \in C^1(\R^d; \R)$ satisfies the following two conditions for all $z \in \mathcal{Z}$.
	\begin{itemize}
		\item[(1)] $(m, b)$-dissipative for some $m, b > 0$. Here, $H \in C^1(\R^d; \R)$ is said to be $(m, b)$-dissipative when the following inequality holds. 
		\begin{align}
			\label{Def_of_Dissipative}
			\langle \nabla H(x), x \rangle_{\R^d} 
			\geq m \| x \|_{\R^d}^2 - b,\qquad x \in \R^d.
		\end{align} 
		\item[(2)] $M$-smooth for some $M > 0$. Here, $H \in C^1(\R^d; \R)$ is said to be $M$-smooth when the following inequality holds. 
		\begin{align}
			\label{Def_of_M-smooth}
			\| \nabla H(x) - \nabla H(y) \|_{\R^d} 
			\leq M \| x - y \|_{\R^d},\qquad x, y \in \R^d.
		\end{align} 
	\end{itemize}
\end{assu}

As in previous works \citep{Mou, Ragi, Xu, Zhang}, we decompose the first term in the R.H.S of (\ref{SGLD_Err_Decompose}) to the sum of $E[L(X_t^{(n, \eta, B)})] - E[L(X_t^{(n, \eta)})]$, $E[L(X_t^{(n, \eta)})] - E[L(X_t^{(n)})]$ and $E[L(X_t^{(n)})] - E[L_n(X_t^{(n)})]$, and prove the following bounds. 
Here, $f = O_\alpha(g)$ means that there exists a constant $C_\alpha > 0$ depends only on $\alpha$ such that $f \leq C_\alpha g$. 

\begin{thm}
	\label{Main_Theorem_SGLD}
	Under Assumption \ref{Assum_Main2}, for some $\eta_0 = O_{m, M}(1)$, the following inequalities hold uniformly on $0 < \eta \leq \eta_0$. 
	Here, $\alpha_0 = (m, b, M, \beta, A, E[\| X_0 \|_{\R^d}^4], d)$. 
	\begin{itemize}
		\item[(1)] $| E[L(X_t^{(n)})] - E[L_n(X_t^{(n)})] | \leq O_{\alpha_0}(n^{-1})$,
		\item[(2)] $| E[L(X_t^{(n, \eta)})] - E[L(X_t^{(n)})] | \leq O_{\alpha_0}(\eta^{1/2})$,
		\item[(3)] $| E[L(X_t^{(n, \eta, B)})] - E[L(X_t^{(n, \eta)})] | \leq O_{\alpha_0}(\eta^{1/2} + \sqrt{(n - B) / B (n - 1)})$.
	\end{itemize}
\end{thm}

Since $t$ is large, Theorem \ref{Main_Theorem_SGLD} (2) and (3) are refinements of Corollary 2.9 in \citep{Zhang} and Theorem 3.6 in \citep{Xu}, which are the sharpest evaluation on $\eta$ and $B$ among previous works, respectively. 
On the other hand, a bound equivalent to Theorem \ref{Main_Theorem_SGLD} (1) has already been shown in Theorem 1 in \citep{Ragi}, which is the sharpest evaluation on $n$ among previous works. 
However, in this paper, we prove Theorem \ref{Main_Theorem_SGLD} (1), (2), and (3) by a single same method based ARC, while previous works \citep{Mou, Ragi, Xu, Zhang} have derived bounds for $n$, $\eta$ and $B$ in individual ways. 

\section{Proof of First Main Result}
\label{SEC_Main1_Proof}

In this section, we prove our first main result, Theorem \ref{Main_Theorem_Weak_Conv}. 
The scheme of our proof is as follows. 
First, we prove the tightness of $(X, Y^{(\varepsilon)})$ and the existence of its weak limit $(\tilde{X}, \tilde{Y})$. 
Next, we confirm that $\tilde{Y}$ solves the martingale problem corresponding to (\ref{Functional_SDE}) and the law of $\tilde{Y}$ coincides with that of $Y$. 
Therefore, the weak limit $(\tilde{X}, \tilde{Y})$ defines a coupling between $\mathcal{L}(X)$ and $\mathcal{L}(Y)$, and Theorem \ref{Main_Theorem_Weak_Conv} is proved. 

\subsection{Auxiliary lemmas}

To prove Thorem \ref{Main_Theorem_Weak_Conv}, we prepare the following three lemmas assuming Assumption \ref{Assum_Main1}. 


\begin{lem}
	\label{Lp_Bounded_Y}
	There exists the strong solution $(X, Y^{(\varepsilon)})$ of (\ref{General_Approximate_Reflection_Coupling}) uniquely and $\sup_{\varepsilon > 0} E[\sup_{0 \leq s \leq t} \| (X_s, Y_s^{(\varepsilon)}) \|_{\R^{2 d}}^p] < \infty$ holds for each fixed $t > 0$. 
\end{lem}


\begin{proof}
	The existence and uniqueness are standard.
	The latter claim can be proved by Gronwall's lemma. 
\end{proof}


\begin{lem}
	\label{Tightness_of_Y}
	The family $\{ (X, Y^{(\varepsilon)}) \}_{\varepsilon > 0}$ is tight. 
\end{lem}


\begin{proof}
	According to Theorem 2.4.10 and Problem 2.4.11 in \citep{kara}, we only have to prove the following two conditions since $X$ is independent of $\varepsilon > 0$. 
	\begin{itemize}
		\setlength{\leftskip}{0.5cm}
		\item[{\bf [T1]}] $\sup_{\varepsilon > 0} E[\| Y_0^{(\varepsilon)} \|_{\R^d}^\nu] < \infty$,
		\item[{\bf [T2]}]$\sup_{\varepsilon > 0} E[\| Y_t^{(\varepsilon)} - Y_s^{(\varepsilon)} \|_{\R^d}^q] \leq C_T (t -s )^{1 + r}$,\qquad $T > 0$,\quad $0 \leq s \leq t \leq T$.  
	\end{itemize}
	Here, $\nu > 0$ and $q, r > 0$ are constants independent of $T$, while $C_T > 0$ may depend on $T$. 
	
	For {\bf [T1]}, we can take $\nu = p$ by Lemma \ref{Lp_Bounded_Y}. 
	For {\bf [T2]}, we have 
	\begin{align*}
		E[\| Y_t^{(\varepsilon)} - Y_s^{(\varepsilon)} \|_{\R^d}^p]
		&\leq 2^p \left\{ (t-s)^{\frac{p}{p-1}} \int_s^t E[\|G(u, Y_u^{(\varepsilon)})\|_{\R^d}^p] du+ O_{\sigma, d, p} \left( (t - s)^{p/2}  \right) \right\}
	\end{align*}
	by Burkholder--Davis--Gundy inequality. 
	Thus, {\bf [T2]} holds for $q = p$, $r = \min \{ p/2 - 1, p / (p-1) \}$ by Lemma \ref{Lp_Bounded_Y}.
\end{proof}


\begin{lem}
	\label{Vanish_Occupation_Time}
	For any subsequence $\varepsilon_\ell$, we can extract a further subsequence $\varepsilon_{\ell_k}$ so that 
	\begin{align*}
		\lim_{k \to \infty} \{1 - h_{\varepsilon_{\ell_k}}(\| Z_t^{(\varepsilon_{\ell_k})} \|_{\R^d})\} h_{\varepsilon_{\ell_k}}(\| Z_t^{(\varepsilon_{\ell_k})} \|_{\R^d}) 
		= 0
	\end{align*}
	holds almost everywhere on $[0, \infty) \times \Omega$. 
	Here, $Z_t^{(\varepsilon)} = X_t - Y_t^{(\varepsilon)}$. 
\end{lem}


\begin{proof}
	Since $e_{i, s}^{(\varepsilon)} = Z_{i, s}^{(\varepsilon)} / \|Z_s^{(\varepsilon)}\|_{\R^d}$, for all $\delta > 0$, we have
	\begin{align*}
		\sum_{i, j=1}^d \left( \frac{\delta_{ij}}{(\|Z_s^{(\varepsilon)}\|_{\R^d}^2 + \delta)^{1/2}} - \frac{Z_{i, s}^{(\varepsilon)} Z_{j, s}^{(\varepsilon)}}{(\|Z_s^{(\varepsilon)}\|_{\R^d}^2 + \delta)^{3/2}} \right) e_{i, s}^{(\varepsilon)} e_{j, s}^{(\varepsilon)} 
		&= \frac{\delta}{(\|Z_s^{(\varepsilon)}\|_{\R^d}^2 + \delta)^{3/2}}. 
	\end{align*}
	Here, $\delta_{i j}$ denotes the Kronecker's delta. 
	Thus, Ito's formula yields
	\begin{align*}
		(\|Z_t^{(\varepsilon)}\|_{\R^d}^2 + \delta)^{1/2} 
		&= (\|Z_0^{(\varepsilon)}\|_{\R^d}^2 + \delta)^{1/2} 
		+ \int_0^t \frac{1}{(\|Z_s^{(\varepsilon)}\|_{\R^d}^2 + \delta)^{1/2}} \langle Z_s^{(\varepsilon)}, dV_s - G(s, Y^{(\varepsilon)}) ds \rangle_{\R^d} \\
		&\quad+ 2 \sigma \int_0^t \frac{\|Z_s^{(\varepsilon)}\|_{\R^d} h_\varepsilon(\|Z_s^{(\varepsilon)}\|_{\R^d})}{(\|Z_s^{(\varepsilon)}\|_{\R^d}^2 + \delta)^{1/2}} \langle e_s^{(\varepsilon)}, dW_s \rangle_{\R^d} 
		+ 2 \sigma^2 \int_0^t \frac{\delta h_\varepsilon(\| Z_s^{(\varepsilon)} \|)^2}{(\|Z_s^{(\varepsilon)}\|_{\R^d}^2 + \delta)^{3/2}} ds.  
	\end{align*} 
	Since $h_\varepsilon$ is not $0$ only on $\{ |a| \geq \varepsilon \}$, 
	\[
	\frac{Z_s^{(\varepsilon)}}{(\|Z_s^{(\varepsilon)}\|_{\R^d}^2 + \delta)^{1/2}} \to \chi_{\{ Z_s^{(\varepsilon)} \neq 0 \}} e_s^{(\varepsilon)},\qquad
	\frac{\| Z_s^{(\varepsilon)} \|_{\R^d} h_\varepsilon(\| Z_s^{(\varepsilon)} \|_{\R^d})}{(\|Z_s^{(\varepsilon)}\|_{\R^d}^2 + \delta)^{1/2}} \to h_\varepsilon(Z_s^{(\varepsilon)}),\qquad
	\frac{\delta h_\varepsilon(\| Z_s^{(\varepsilon)} \|_{\R^d})^2}{(\|Z_s^{(\varepsilon)}\|_{\R^d}^2 + \delta)^{3/2}} \leq \frac{\delta}{(\varepsilon^2 + \delta)^{3/2}} \to 0
	\]
	hold as $\delta \to 0$. 
	Therefore, taking the limit $\delta \to 0$, we find that $\| Z_t^{(\varepsilon)} \|_{\R^d}$ is a one-dimensional semi-martingale satisfying
	\begin{align*}
		\| Z_t^{(\varepsilon)} \|_{\R^d} 
		&= \| Z_0^{(\varepsilon)} \|_{\R^d}
		+ \int_0^t \chi_{\{ Z_s^{(\varepsilon)} \neq 0 \}} \langle e_s^{(\varepsilon)}, dV_s - G(s, Y^{(\varepsilon)}) ds \rangle_{\R^d} 
		+ 2 \sigma \int_0^t h_\varepsilon(\| Z_s^{(\varepsilon)} \|_{\R^d}) \langle e_s^{(\varepsilon)}, dW_s \rangle_{\R^d} \\
		&\eqqcolon \| Z_0^{(\varepsilon)} \|_{\R^d} + V_t^{(\varepsilon)} + M_t^{(\varepsilon)}.
	\end{align*} 
	According to (3.7.10) in \citep{kara}, the local time $\Lambda^{(\varepsilon)}(t, a)$ of $\| Z_t^{(\varepsilon)} \|_{\R^d}$ is given by
	\[
	\Lambda^{(\varepsilon)}(t, x) 
	= |\| Z_t^{(\varepsilon)} \|_{\R^d} -  a| - |\| Z_0^{(\varepsilon)} \|_{\R^d} - a| - \int_0^t {\rm sgn}(\| Z_s^{(\varepsilon)} \|_{\R^d} -  a) d M_s^{(\varepsilon)} - \int_0^t {\rm sgn}(\| Z_s^{(\varepsilon)} \|_{\R^d} -  a) d V_s^{(\varepsilon)}.
	\]
	In particular, for each fixed $t$, $\sup_{\varepsilon > 0, |a| \leq 1} E[\Lambda^{(\varepsilon)}(t, a)] < \infty$ holds by Lemma \ref{Lp_Bounded_Y}. 
	In addition, by the definition of the local time, we have
	\begin{align}
		\label{Occupation_Density_Formula}
		\int_0^t v(\| Z_s^{(\varepsilon)} \|_{\R^d}) d \langle M^{(\varepsilon)} \rangle_s 
		= \int_{\R} v(a) \Lambda^{(\varepsilon)}(t, a) da 
	\end{align}
	for all measurable function $v : \R \to [0, \infty)$. 
	Taking $v = 1 - h_\varepsilon$ in (\ref{Occupation_Density_Formula}), since $1 - h_\varepsilon$ is not $0$ only on $\{ |a| \leq 2 \varepsilon \}$, we obtain
	\[
	\int_0^t \{ 1 - h_\varepsilon(\| Z_s^{(\varepsilon)} \|_{\R^d}) \} h_\varepsilon(\| Z_s^{(\varepsilon)} \|_{\R^d})^2 d s 
	= \frac{1}{4 \sigma^2} \int_{\R} \{ 1 - h_\varepsilon(a) \} \Lambda^{(\varepsilon)}(t, a) da 
	= \frac{1}{4 \sigma^2} \int_{-2 \varepsilon}^{2 \varepsilon} \{ 1 - h_\varepsilon(a) \} \Lambda^{(\varepsilon)}(t, a) da. 
	\]
	Thus, 
	\[
	\varlimsup_{\ell \to \infty} E[\int_0^t \{ 1 - h_{\varepsilon_\ell}(\| Z_s^{(\varepsilon_\ell)} \|_{\R^d}) \} h_\varepsilon(\| Z_s^{(\varepsilon_\ell )}\|_{\R^d})^2 d s ] 
	\leq \varlimsup_{\ell \to \infty} \frac{\varepsilon_\ell}{\sigma^2} \sup_{\varepsilon > 0, |a| \leq 2 \varepsilon} E[\Lambda^{(\varepsilon)}(t, a)] 
	= 0
	\]
	holds and we can extract a subsequence $\varepsilon_{\ell_k}$ so that $\{ 1 - h_{\varepsilon_{\ell_k}}(\| Z_s^{(\varepsilon_{\ell_k})} \|_{\R^d}) \} h_{\varepsilon_{\ell_k}}(\| Z_s^{(\varepsilon_{\ell_k})}\|_{\R^d}) \to 0$ holds almost everywhere on $[0, t] \times \Omega$. 
	Applying the diagonal argument, we obtain the desired result. 
\end{proof}

\subsection{Proof of Theorem \ref{Main_Theorem_Weak_Conv}}

By Lemma \ref{Tightness_of_Y} and Prohorov's theorem, we can extract a subsequence $\varepsilon_\ell$ so that $(X, Y^{(\varepsilon_\ell)})$ has its weak limit $(\tilde{X}, \tilde{Y})$. 
To confirm that $\mathcal{L}(\tilde{X}, \tilde{Y})$ is a coupling between $\mathcal{L}(X)$ and $\mathcal{L}(Y)$, we only have to show $\mathcal{L}(\tilde{Y}) = \mathcal{L}(Y)$ since $\mathcal{L}(\tilde{X}) = \mathcal{L}(X)$ is obvious. 
Furthermore, the pathwise uniqueness of (\ref{Functional_SDE}) yields the weak uniqueness of itself by corollary to Lemma 5.1.2 in \citep{Ikewata}. 
Therefore, according to Propositions 5.4.6 and 5.4.11 in \citep{kara}, we only have to prove that $M^f$ defined by (\ref{Martingale_Problem}) below is a martingale for all compact supported $f \in C^2(\R^d; \R)$. 
\begin{align}
	\label{Martingale_Problem}
	M_t^f 
	\coloneqq f(\tilde{Y}_t) - f(\tilde{Y}_0) - \int_0^t \langle G(s, \tilde{Y}), \nabla f(\tilde{Y}_s) \rangle_{\R^d} ds - \frac{\sigma^2}{2} \int_0^t \Delta f(\tilde{Y}_s) ds.
\end{align}
For each $s \geq 0$, let $\B_s$ be the smallest $\sigma$-algebra on $C([0, \infty); \R^d)$ such that the map $C([0, \infty); \R^d) \ni \varphi \mapsto \varphi(\min \{ \cdot, s \})$ is $\B_s$-measurable. 
By Ito's formula, 
\begin{align}
	M_t^{(\varepsilon, f)} 
	&\coloneqq f(Y^{(\varepsilon)}_t) 
	- f(Y^{(\varepsilon)}_0) 
	- \int_0^t \langle G(s, Y^{(\varepsilon)}), \nabla f(Y^{(\varepsilon)}_s) \rangle_{\R^d} ds - \frac{\sigma^2}{2} \int_0^t \Delta f(Y^{(\varepsilon)}_s) ds \notag \\
	&\quad + 2 \sigma^2 \sum_{i, j=1}^d \int_0^t \partial_{i j}^2 f(Y^{(\varepsilon)}_s) \{1 - h_{\varepsilon}(\| Z_s^{(\varepsilon)} \|_{\R^d})\} h_{\varepsilon}(\| Z_s^{(\varepsilon)} \|_{\R^d}) e_{i, s}^{(\varepsilon)} e_{j, s}^{(\varepsilon)} ds \label{Approx_Martingale_Problem}.
\end{align}
is a martingale for all $\varepsilon > 0$. 
Thus, for all $s \leq t$ and $\B_s$-measurable bounded continuous functional $F : C([0, \infty); \R^d) \to \R$, we have
\begin{align}
	\label{Sol_Approx_Martingale_Problem}
	E[M_t^{(\varepsilon, f)} F(Y^{(\varepsilon)})] 
	= E[M_s^{(\varepsilon, f)} F(Y^{(\varepsilon)})]. 
\end{align}
By Lemma \ref{Vanish_Occupation_Time}, extracting a further subsequence of $\varepsilon_\ell$ if needed, we may assume that $\{1 - h_{\varepsilon_\ell}(\| Z_s^{(\varepsilon_\ell)} \|_{\R^d})\} h_{\varepsilon_\ell}(\| Z_s^{(\varepsilon_\ell)} \|_{\R^d}) \to 0$ holds almost everywhere. 
As a result, taking the limit $\varepsilon \to 0$ along with $\varepsilon_\ell$ in (\ref{Sol_Approx_Martingale_Problem}), Lemma \ref{WeakConv_Non_Bdd} yields
\begin{align}
	\label{SolMartingale_Problem}
	E[M_t^f F(\tilde{Y})] 
	= E[M_s^f F(\tilde{Y})]. 
\end{align}
(\ref{SolMartingale_Problem}) means that $M_t^f$ is a martingale. 
\qed

\section{Bounds on the Difference between Langevin Dynamics with Different Drift Terms}
\label{SEC_First_Appli}

In this section, we describe the first application of Theorem \ref{Main_Theorem_Weak_Conv}. 
We apply Theorem \ref{Main_Theorem_Weak_Conv} to the evaluation of the difference between Langevin dynamics along with gradients of $F, G \in C^1(\R^d; \R)$, and prove Theorem  \ref{Main_Theorem_SGLD} (1) as its corollary. 
Here, $F$ and $G$ are assumed to be $(m, b)$-dissipative and $M$-smooth. 

Let $X$ and $Y$ be the solutions of SDEs with initial values $X_0$ and $Y_0$
\begin{align}
	&dX_t = - \nabla F(X_t) dt + \sqrt{2 / \beta} dW_t, \label{LSDE_along_F}\\
	&dY_t = - \nabla G(Y_t) dt + \sqrt{2 / \beta} dW_t, \label{LSDE_along_G}
\end{align}
respectively. 
First, adopting the technique developed in \citep{Ebe}, we derive a bound on the difference between $\mathcal{L}(X_t)$ and $\mathcal{L}(Y_t)$ on the bases of the ARC $(X, Y^{(\varepsilon)})$ of $(X, Y)$ defined by
\begin{align}
	\label{GenBound_Approximate_Reflection_Coupling}
	\begin{cases}
		dY_t^{(\varepsilon)} 
		= - \nabla G(Y_t^{(\varepsilon)}) dt + \sqrt{2 / \beta} (I_d - 2 h_\varepsilon( \| X_t - Y_t^{(\varepsilon)} \|_{\R^d} ) e_t^{(\varepsilon)} {e_t^{(\varepsilon)}}^\top) d W_t, \\
		Y_0^{(\varepsilon)} = Y_0. 
	\end{cases}
\end{align}
Here, $e_t^{(\varepsilon)} = (X_t - Y_t^{(\varepsilon)}) / \| X_t - Y_t^{(\varepsilon)} \|_{\R^d}$. 

\subsection{Notations form \citep{Ebe}}
\label{SUBSEC_Ebe_Notation}

Before going into the details, we prepare the special case of notations used in \citep{Ebe}. 
For $p > 0$, define $V_p : \R^d \to \R$ by $V_p(x) = \| x \|_{\R^d}^p$ and let $\bar{V}_p(x) = 1 + V_p(x)$.
For constants $C(p)$ and $\lambda(p)$ defined in Lemma \ref{p-th_Lyap}, let
\begin{align}
	\label{Def_C_Lambda}
	C = C(2) + \lambda(2),\qquad 
	\lambda = \lambda(2)
\end{align}
and let 
\begin{align}
	S_1 
	&\coloneqq \{ (x, y) \in \R^d \times \R^d \mid \bar{V}_2(x) + \bar{V}_2(y) \leq 2 \lambda^{-1} C \}, \label{DefOfS1} \\
	S_2
	&\coloneqq \{ (x, y) \in \R^d \times \R^d \mid \bar{V}_2(x) + \bar{V}_2(y) \leq 4C(1+\lambda^{-1}) \}. \label{DefOfS2} 
\end{align}
The diameters of $S_1$ and $S_2$ are denoted by $R_1$ and $R_2$, respectively, where the diameter of a set $\Gamma \subset \R^d$ is defined by $\sup_{x, y \in \Gamma} \| x - y \|_{\R^d}$. 

For a constant $\kappa$ defined by (\ref{Condi_On_kappa}), we define $Q(\kappa)$ by
\begin{align}
	\label{DefOfQ}
	Q(\kappa)
	\coloneqq \sup_{x \in \R^d} \frac{\| \nabla \bar{V}_2(x) \|_{\R^d}}{\max \{ \bar{V}_2(x), \kappa^{-1} \}} 
	= \sup_{x \in \R^d} \frac{2 \| x \|_{\R^d}}{\max \{ 1 + \| x \|_{\R^d}^2, \kappa^{-1} \}} 
	= 2 \sqrt{ \kappa - \kappa^2 } \in (0, 1]
\end{align}
and functions $\varphi, \Phi : [0, \infty) \to [0, \infty)$ by
\begin{align}
	\label{DefOfPhi}
	\varphi(r) 
	\coloneqq \exp \left( - \frac{M \beta}{8} r^2 - 2 Q(\kappa) r \right),\quad
	\Phi(r) = \int_0^r \varphi(s) ds, 
\end{align}
respectively. 
For constants $\zeta$ and $\xi$ defined by
\begin{align}
	\label{DefOfBetaAndXi}
	\frac{1}{\zeta} 
	\coloneqq \int_0^{R_2} \Phi(s) \varphi(s)^{-1} ds,\quad
	\frac{1}{\xi} 
	\coloneqq \int_0^{R_1} \Phi(s) \varphi(s)^{-1} ds, 
\end{align}
let 
\begin{align}
	\label{DefOfG}
	g(r) 
	\coloneqq 1 - \frac{\zeta}{4} \int_0^{\min \{r, R_2\}} \Phi(s) \varphi(s)^{-1} ds - \frac{\xi}{4} \int_0^{\min \{r, R_1\}} \Phi(s) \varphi(s)^{-1} ds.
\end{align}
Furthermore, for
\begin{align}
	\label{Def_Of_F}
	f(r) 
	\coloneqq 
	\begin{cases}
		{\displaystyle\int_0^{\min \{r, R_2\}} \varphi(s) g(s) ds}, &r \geq 0 \\
		r, &r<0
	\end{cases}
\end{align}
and $U(x, y) \coloneqq 1 + \kappa \bar{V}_2(x) + \kappa \bar{V}_2(y)$, let
\begin{align}
	\label{Rho2}
	\rho_2(x, y) 
	= f(\| x-y \|_{\R^d}) U(x, y),\qquad x, y \in \R^d.
\end{align}
Finally, for probability measures $\mu$ and $\nu$ on $\R^d$, denoting the set of all couplings between them by $\Pi(\mu, \nu)$, let
\begin{align}
	\label{W_rho2}
	\mathcal{W}_{\rho_2}(\mu, \nu) 
	\coloneqq \inf_{\gamma \in \Pi(\mu, \nu)} \int_{\R^d \times \R^d} \rho_2(x, y) \gamma(dx dy).
\end{align}
Here, for random variables $Z_1$ and $Z_2$, $\mathcal{W}_{\rho_2}(\mathcal{L}(Z_1), \mathcal{L}(Z_2))$ may be abbreviated as $\mathcal{W}_{\rho_2}(Z_1, Z_2)$. 

\subsection{Uniform bound on the time for the difference between (\ref{LSDE_along_F}) and (\ref{LSDE_along_G})} 

The following Proposition \ref{Main_Theorem_GenBound} is an extension of Theorem 2.2 in \citep{Ebe} to the case of different drift terms. 
Although Proposition \ref{Main_Theorem_GenBound} is proved in the same manner as Theorem 2.2 in \citep{Ebe}, we give the complete proof of it to explain that the error terms caused by adopting ARC do not disturb the proof. 


\begin{prop}
	\label{Main_Theorem_GenBound}
	Let $X_0, Y_0 \in L^4(\Omega; \R^d)$ and let $c \coloneqq \min \left\{ \frac{\zeta}{\beta}, \frac{\lambda}{2}, 2 C \lambda \kappa \right\}$.
	Then, for the solutions $X$ and $Y$ of (\ref{LSDE_along_F}) and (\ref{LSDE_along_G}), we have 
	\begin{align}
		\mathcal{W}_{\rho_2}(X_t, Y_t)
		&\leq e^{- c t} E\rho_2(X_0, Y_0)] 
		+ O_\alpha \left( e^{- c t} \int_0^t e^{c s} E[\| \nabla F(Y_s) - \nabla G(Y_s) \|_{\R^d}^2]^{1/2} ds \right). \label{Main_Result_GenBound}
	\end{align}
	Here, $\alpha = (m, b, M, \beta, \|\nabla F(0)\|_{\R^d}, \|\nabla G(0)\|_{\R^d}, E[\| X_0 \|_{\R^d}^4], E[\| Y_0 \|_{\R^d}^4], d)$. 
\end{prop}


\begin{proof}
	Fix $\varepsilon > 0$ and let $Z_t^{(\varepsilon)} = X_t - Y_t^{(\varepsilon)}$ for $X_t$ and $Y_t^{(\varepsilon)}$ defined by (\ref{GenBound_Approximate_Reflection_Coupling}). \\
	
	
	\noindent {\bf \underline{Step 1} }\ Evaluation of $f(\| Z_t^{(\varepsilon)} \|_{\R^d})$. \\
	
	\noindent As in the proof of Lemma \ref{Vanish_Occupation_Time}, we can show that $\| Z_t^{(\varepsilon)} \|_{\R^d}$ is an one-dimensional semi-martingale satisfying
	\begin{align*}
		\| Z_t^{(\varepsilon)} \|_{\R^d} 
		&= \| Z_0^{(\varepsilon)} \|_{\R^d}
		- \int_0^t \chi_{\{ Z_s^{(\varepsilon)} \neq 0 \}} \langle e_s^{(\varepsilon)}, \nabla F(X_s) - \nabla G(Y_s^{(\varepsilon)}) \rangle_{\R^d} ds
		+ 2 \sqrt{\frac{2}{\beta}} \int_0^t h_\varepsilon(\| Z_s^{(\varepsilon)} \|_{\R^d}) \langle e_s^{(\varepsilon)}, dW_s \rangle_{\R^d}. 
	\end{align*}
	Let $\Lambda^{(\varepsilon)}(t, a)$ be the local time of $\| Z^{(\varepsilon)} \|_{\R^d}$. 
	Since $f$ is a concave function by Lemma \ref{PropertyOfF}, Tanaka's formula (Theorem 3.7.1 in \citep{kara}) yields
	\begin{align*}
		f(\| Z_t^{(\varepsilon)} \|_{\R^d}) -f(\| Z_0^{(\varepsilon)} \|_{\R^d})
		&= - \int_0^t f_-^\prime(\| Z_s^{(\varepsilon)} \|_{\R^d}) \chi_{\{ Z_s^{(\varepsilon)} \neq 0 \}} \langle e_s^{(\varepsilon)}, \nabla F(X_s) - \nabla G(Y_s^{(\varepsilon)}) \rangle_{\R^d} ds \\
		&\quad+ 2 \sqrt{\frac{2}{\beta}} \int_0^t f_-^\prime(\| Z_s^{(\varepsilon)} \|_{\R^d}) h_\varepsilon(\| Z_s^{(\varepsilon)} \|_{\R^d}) \langle e_s^{(\varepsilon)}, dW_s \rangle_{\R^d} 
		+ \frac{1}{2} \int_{-\infty}^{\infty} \Lambda^{(\varepsilon)}(t, a) \mu_f(da).
	\end{align*}
	Here, $f_-^\prime$ and $\mu_f$ denote the left derivative and the second derivative measure of $f$, respectively. 
	Whereas, by the definition of the local time, for a measurable function $v = \chi_{\{ R_1, R_2 \}}$, we have
	\begin{align}
		\label{Local_Time_Has_No_Mass_At_Points}
		\frac{8}{\beta} \int_0^t \chi_{\{ R_1, R_2 \}}(\| Z_s^{(\varepsilon)} \|_{\R^d}) h_\varepsilon( \| Z_s^{(\varepsilon)} \|_{\R^d} )^2 ds
		= \int_{\R} \chi_{\{ R_1, R_2 \}}(a) \Lambda^{(\varepsilon)}(t, a) da
		= 0.
	\end{align}
	Since $h_\varepsilon( a )$ is equal to $1$ if $| a | \geq 2 \varepsilon$, the occupation time of $\| Z^{(\varepsilon)} \|_{\R^d}$ on $\{ R_1, R_2 \}$ must be $0$ by (\ref{Local_Time_Has_No_Mass_At_Points}) if $\varepsilon$ is sufficiently small for $R_1$ and $R_2$. 
	Thus, we may assume $f \in C^2(\R; \R)$ when we consider its value at $\| Z_s^{(\varepsilon)} \|_{\R^d}$. 
	
	Similarly, since $\mu_f(\{ R_1, R_2 \}) \leq 0$ by Lemma \ref{PropertyOfMu_f}, the definition of the local time yields
	\begin{align*}
		\int_{-\infty}^\infty \Lambda^{(\varepsilon)}(t, a) \mu_f(da)
		&\leq \int_{-\infty}^\infty \chi_{\R \setminus \{ R_1, R_2 \}}(a) \Lambda^{(\varepsilon)}(t, a) \mu_f(da) \\
		&= \int_{-\infty}^\infty \chi_{\R \setminus \{ R_1, R_2 \}}(a) f^{\prime \prime}(a) \Lambda^{(\varepsilon)}(t, a) da \\
		&= \frac{8}{\beta} \int_0^t f^{\prime \prime}(\| Z_s^{(\varepsilon)} \|_{\R^d}) h_\varepsilon( \| Z_s^{(\varepsilon)} \|_{\R^d} )^2 ds.
	\end{align*}
	Therefore, we obtain 
	\begin{align*}
		d f(\| Z_t^{(\varepsilon)} \|_{\R^d})
		&\leq - f^\prime(\| Z_t^{(\varepsilon)} \|_{\R^d}) \chi_{\{ Z_t^{(\varepsilon)} \neq 0 \}} \langle e_t^{(\varepsilon)}, \nabla F(X_t) - \nabla G(Y_t^{(\varepsilon)}) \rangle_{\R^d} dt \\
		&\quad+ 2 \sqrt{\frac{2}{\beta}} f^\prime(\| Z_t^{(\varepsilon)} \|_{\R^d}) h_\varepsilon(\| Z_t^{(\varepsilon)} \|_{\R^d}) \langle e_t^{(\varepsilon)}, dW_t \rangle_{\R^d} 
		+ \frac{4}{\beta} f^{\prime \prime}(\| Z_t^{(\varepsilon)} \|_{\R^d}) h_\varepsilon( \| Z_t^{(\varepsilon)} \|_{\R^d} )^2 dt. 
	\end{align*}
	
	Finally, by $M$-smoothness of $F$, we have
	\begin{align}
		\label{M_smooth_Bound}
		\langle e_t^{(\varepsilon)}, \nabla F(X_t) - \nabla F(Y_t^{(\varepsilon)}) \rangle_{\R^d}
		\leq M \| Z_t^{(\varepsilon)} \|_{\R^d}. 
	\end{align}
	Furthermore, by Lemma \ref{Esti_On_f}, we have also
	\begin{align*}
		f^{\prime \prime}(\| Z_t^{(\varepsilon)} \|_{\R^d}) 
		&\leq - \left( \frac{M \beta}{4} \| Z_t^{(\varepsilon)} \|_{\R^d} + 2Q(\kappa) \right) f^\prime(\| Z_t^{(\varepsilon)} \|_{\R^d}) 
		- \frac{\zeta}{4} f(\| Z_t^{(\varepsilon)} \|_{\R^d}) \chi_{(0, R_2)}(\| Z_t^{(\varepsilon)} \|_{\R^d}) 
		- \frac{\xi}{4} f(\| Z_t^{(\varepsilon)} \|_{\R^d}) \chi_{(0, R_1)}(\| Z_t^{(\varepsilon)} \|_{\R^d}). 
	\end{align*}
	As a result, noting $0 \leq f^\prime \leq 1$, we obtain
	\begin{align}
		d f(\| Z_t^{(\varepsilon)} \|_{\R^d}) 
		&\leq \| \nabla F(Y_t^{(\varepsilon)}) - \nabla G(Y_t^{(\varepsilon)}) \|_{\R^d} dt \notag
		+ M \| Z_t^{(\varepsilon)} \|_{\R^d} \{1 -  h_\varepsilon( \| Z_t^{(\varepsilon)} \|_{\R^d} )^2\} dt \notag \\
		&\quad+ 2 \sqrt{\frac{2}{\beta}} f^\prime(\| Z_t^{(\varepsilon)} \|_{\R^d}) h_\varepsilon(\| Z_t^{(\varepsilon)} \|_{\R^d}) \langle e_t^{(\varepsilon)}, dW_t \rangle_{\R^d} 
		- \frac{8Q(\kappa)}{\beta} f^\prime(\| Z_t^{(\varepsilon)} \|_{\R^d}) h_\varepsilon( \| Z_t^{(\varepsilon)} \|_{\R^d} )^2 dt \notag \\
		&\quad- \frac{\zeta}{\beta} f(\| Z_t^{(\varepsilon)} \|_{\R^d}) \chi_{(0, R_2)}(\| Z_t^{(\varepsilon)} \|_{\R^d}) dt
		- \frac{\xi}{\beta} f(\| Z_t^{(\varepsilon)} \|_{\R^d}) \chi_{(0, R_1)}(\| Z_t^{(\varepsilon)} \|_{\R^d}) dt. \label{Ebe_Step1}
	\end{align}
	
	
	\noindent {\bf \underline{Step 2} }\ Evaluation of $U(X_t, Y_t^{(\varepsilon)})$. \\
	
	\noindent Ito's formula yields
	\begin{align}
		d U(X_t, Y_t^{(\varepsilon)}) 
		&= \kappa (\mathcal{L}_F \bar{V}_2(X_t) + \mathcal{L}_G \bar{V}_2(Y_t^{(\varepsilon)})) dt 
		- \frac{4 \kappa h_\varepsilon( \| Z_t^{(\varepsilon)} \|_{\R^d}) (1 - h_\varepsilon( \| Z_t^{(\varepsilon)} \|_{\R^d} ))}{\beta} \sum_{i, j =1}^d e_{i, t}^{(\varepsilon)} e_{j, t}^{(\varepsilon)} \partial_{i, j}^2 \bar{V}_2(Y_t^{(\varepsilon)}) dt \notag \\
		&\quad+ \kappa \sqrt{\frac{2}{\beta}} \langle \nabla \bar{V}_2(X_t) + \nabla \bar{V}_2(Y_t^{(\varepsilon)}), dW_t \rangle_{\R^d} 
		-2 \kappa \sqrt{\frac{2}{\beta}} h_\varepsilon( \| Z_t^{(\varepsilon)} \|_{\R^d} ) \langle e_t^{(\varepsilon)}, \nabla \bar{V}_2(Y_t^{(\varepsilon)}) \rangle_{\R^d} \langle e_t^{(\varepsilon)}, dW_t \rangle_{\R^d}. \label{Ito_Formula_For_U}
	\end{align}
	Here, $\mathcal{L}_H = - \langle \nabla H, \nabla \rangle_{\R^d} - \beta^{-1} \Delta$, for $H \in C^1(\R^d; \R)$. 
	According to Lemma \ref{IneqDeriveFrom_S}, by the $(m, b)$-dissipativity of $F$ and $G$, we have
	\[
	\mathcal{L}_F \bar{V}_2(X_t) + \mathcal{L}_G \bar{V}_2(Y_t^{(\varepsilon)})
	\leq 2C - \lambda (\bar{V}_2(X_t^{(n)}) + \bar{V}_2(Y_t^{(\varepsilon)})).
	\]
	Furthermore, by Lemma \ref{Choose_kappa} and (\ref{DefOfBetaAndXi}), we have also
	\[
	2C \kappa 
	\leq \frac{1}{\beta} \left( \int_0^{R_1} \varphi(s)^{-1} \Phi(s) ds \right)^{-1}
	=  \frac{\xi}{\beta}.
	\]
	Thus, (\ref{OutOfS1}) and (\ref{OutOfS2}) yield
	\begin{align*}
		\kappa (\mathcal{L}_F \bar{V}_2(X_t) + \mathcal{L}_G \bar{V}_2(Y_t^{(\varepsilon)})) 
		&= \kappa (\mathcal{L}_F \bar{V}_2(X_t) + \mathcal{L}_G \bar{V}_2(Y_t^{(\varepsilon)})) \chi_{(0, R_1]}(\| Z_t^{(\varepsilon)} \|_{\R^d}) 
		-\frac{\lambda}{2} \min\{ 1, 4 C \kappa \} U(X_t, Y_t^{(\varepsilon)}) \chi_{(R_2, \infty)}(\| Z_t^{(\varepsilon)} \|_{\R^d}) \\
		&\leq \frac{\xi}{\beta} \chi_{(0, R_1]}(\| Z_t^{(\varepsilon)} \|_{\R^d}) -\frac{\lambda}{2} \min\{ 1, 4 C \kappa \} U(X_t, Y_t^{(\varepsilon)}) \chi_{(R_2, \infty)}(\| Z_t^{(\varepsilon)} \|_{\R^d}) \\
		&\leq \frac{\xi}{\beta} U(X_t, Y_t^{(\varepsilon)}) \chi_{(0, R_1]}(\| Z_t^{(\varepsilon)} \|_{\R^d}) -\frac{\lambda}{2} \min\{ 1, 4 C \kappa \} U(X_t, Y_t^{(\varepsilon)}) \chi_{(R_2, \infty)}(\| Z_t^{(\varepsilon)} \|_{\R^d}).
	\end{align*}
	As a result, since the Hessian matrix of $\bar{V}_2$ is nonnegative-definite, we obtain
	\begin{align}
		d U(X_t, Y_t^{(\varepsilon)}) 
		&\leq \left( \frac{\xi}{\beta} U(X_t, Y_t^{(\varepsilon)}) \chi_{(0, R_1]}(\| Z_t^{(\varepsilon)} \|_{\R^d}) -\frac{\lambda}{2} \min\{ 1, 4 C \kappa \} U(X_t, Y_t^{(\varepsilon)}) \chi_{(R_2, \infty)}(\| Z_t^{(\varepsilon)} \|_{\R^d}) \right) dt \notag \\
		&\quad+ \kappa \sqrt{\frac{2}{\beta}} \langle \nabla \bar{V}_2(X_t) + \nabla \bar{V}_2(Y_t^{(\varepsilon)}), dW_t \rangle_{\R^d} 
		-2 \kappa \sqrt{\frac{2}{\beta}} h_\varepsilon( \| Z_t^{(\varepsilon)} \|_{\R^d} ) \langle e_t^{(\varepsilon)}, \nabla \bar{V}_2(Y_t^{(\varepsilon)}) \rangle_{\R^d} \langle e_t^{(\varepsilon)}, dW_t \rangle_{\R^d}. \label{Ebe_Step2}
	\end{align}
	
	
	\noindent {\bf \underline{Step 3} }\ Evaluation of $\rho_2(X_t, Y_t^{(\varepsilon)})$. \\
	
	\noindent According to the representations of $f(\| Z_t^{(\varepsilon)} \|_{\R^d})$ and $U(X_t, Y_t^{(\varepsilon)})$ as semi-martingales, 
	\begin{align*}
		d \langle f(\| Z^{(\varepsilon)} \|_{\R^d}), U(X, Y^{(\varepsilon)}) \rangle_t 
		&= \frac{4 \kappa}{\beta} f^\prime(\| Z_t^{(\varepsilon)} \|_{\R^d}) h_\varepsilon(\| Z_t^{(\varepsilon)} \|_{\R^d})^2 \langle e_t^{(\varepsilon)}, \nabla \bar{V}_2(X_t) - \nabla \bar{V}_2(Y_t^{(\varepsilon)}) \rangle_{\R^d} dt \\
		&\quad+ \frac{4 \kappa}{\beta} f^\prime(\| Z_t^{(\varepsilon)} \|_{\R^d}) \{ 1 - h_\varepsilon( \| Z_t^{(\varepsilon)} \|_{\R^d} ) \} h_\varepsilon(\| Z_t^{(\varepsilon)} \|_{\R^d}) \langle e_t^{(\varepsilon)}, \nabla \bar{V}_2(X_t) + \nabla \bar{V}_2(Y_t^{(\varepsilon)}) \rangle_{\R^d} dt
	\end{align*}
	holds. 
	By the definition of $Q(\kappa)$, for  all $x \neq y$ we have
	\begin{align*}
		\kappa \left\langle \nabla \bar{V}_2(x) - \nabla \bar{V}_2(y), \frac{x-y\ \ \ }{\| x-y \|_{\R^d}} \right\rangle_{\R^d}
		&\leq \kappa \| \nabla \bar{V}_2(x) - \nabla \bar{V}_2(y) \|_{\R^d} \\
		&\leq U(x, y) \left( \frac{\| \nabla \bar{V}_2(x) \|_{\R^d}}{\kappa^{-1} + \bar{V}_2(x)} + \frac{\| \nabla \bar{V}_2(y) \|_{\R^d}}{\kappa^{-1} + \bar{V}_2(y)} \right) \\
		&\leq 2Q(\kappa) U(x, y).
	\end{align*}
	Thus, 
	\begin{align}
		d \langle f(\| Z^{(\varepsilon)} \|_{\R^d}), U(X, Y^{(\varepsilon)}) \rangle_t
		&\leq \frac{8 Q(\kappa)}{\beta} f^\prime(\| Z_t^{(\varepsilon)} \|_{\R^d}) U(X_t, Y_t^{(\varepsilon)}) h_\varepsilon(\| Z_t^{(\varepsilon)} \|_{\R^d})^2 dt \notag \\
		&\quad+ \frac{4 \kappa}{\beta} \{ 1 - h_\varepsilon( \| Z_t^{(\varepsilon)} \|_{\R^d} ) \} h_\varepsilon(\| Z_t^{(\varepsilon)} \|_{\R^d}) \langle e_t^{(\varepsilon)}, \nabla \bar{V}_2(X_t) + \nabla \bar{V}_2(Y_t^{(\varepsilon)}) \rangle_{\R^d} dt \label{Bound_Quadratic_Variation}
	\end{align}
	holds. 
	
	For the first term in the R.H.S of 
	\begin{align}
		\label{Ito_Formula_Rho}
		d (f(\| Z_t^{(\varepsilon)} \|_{\R^d}), U(X_t, Y_t^{(\varepsilon)})) 
		= U(X_t, Y_t^{(\varepsilon)}) d f(\| Z_t^{(\varepsilon)} \|_{\R^d}) + f(\| Z_t^{(\varepsilon)} \|_{\R^d}) d U(X_t, Y_t^{(\varepsilon)}) + d \langle f(\| Z^{(\varepsilon)} \|_{\R^d}), U(X, Y^{(\varepsilon)}) \rangle_t, 
	\end{align}
	the following holds by Step 1.
	\begin{align*}
		&U(X_t, Y_t^{(\varepsilon)}) d f(\| Z_t^{(\varepsilon)} \|_{\R^d}) \\
		&\quad\leq - \frac{\zeta}{\beta} \rho_2(X_t, Y_t^{(\varepsilon)}) \chi_{(0, R_2)}(\| Z_t^{(\varepsilon)} \|_{\R^d}) dt 
		- \frac{\xi}{\beta} \rho_2(X_t, Y_t^{(\varepsilon)}) \chi_{(0, R_1)}(\| Z_t^{(\varepsilon)} \|_{\R^d}) dt \\
		&\qquad+ U(X_t, Y_t^{(\varepsilon)}) \| \nabla F(Y_t^{(\varepsilon)}) - \nabla G(Y_t^{(\varepsilon)}) \|_{\R^d} dt
		- \frac{8Q(\kappa)}{\beta} f^\prime(\| Z_t^{(\varepsilon)} \|_{\R^d}) U(X_t, Y_t^{(\varepsilon)}) h_\varepsilon( \| Z_t^{(\varepsilon)} \|_{\R^d} )^2 dt \\
		&\qquad+2 \sqrt{\frac{2}{\beta}} f^\prime(\| Z_t^{(\varepsilon)} \|_{\R^d}) U(X_t, Y_t^{(\varepsilon)}) h_\varepsilon(\| Z_t^{(\varepsilon)} \|_{\R^d}) \langle e_t^{(\varepsilon)}, dW_t \rangle_{\R^d} 
		+ M U(X_t, Y_t^{(\varepsilon)}) \| Z_t^{(\varepsilon)} \|_{\R^d} \{ 1 -  h_\varepsilon( \| Z_t^{(\varepsilon)} \|_{\R^d} )^2 \} dt. 
	\end{align*}
	Similarly, for the second term in the R.H.S of (\ref{Ito_Formula_Rho}), the following holds by Step 2. 
	\begin{align*}
		&f(\| Z_t^{(\varepsilon)} \|_{\R^d}) d U(X_t, Y_t^{(\varepsilon)}) \\
		&\quad\leq \left( \frac{\xi}{\beta} \rho_2(X_t, Y_t^{(\varepsilon)}) \chi_{(0, R_1]}(\| Z_t^{(\varepsilon)} \|_{\R^d}) -\frac{\lambda}{2} \min\{ 1, 4 C \kappa \} \rho_2(X_t, Y_t^{(\varepsilon)}) \chi_{(R_2, \infty)}(\| Z_t^{(\varepsilon)} \|_{\R^d}) \right) dt \\ 
		&\qquad+ \kappa \sqrt{\frac{2}{\beta}} f(\| Z_t^{(\varepsilon)} \|_{\R^d}) \langle \nabla \bar{V}_2(X_t) + \nabla \bar{V}_2(Y_t^{(\varepsilon)}), dW_t \rangle_{\R^d} 
		-2 \kappa \sqrt{\frac{2}{\beta}} f(\| Z_t^{(\varepsilon)} \|_{\R^d}) h_\varepsilon( \| Z_t^{(\varepsilon)} \|_{\R^d} ) \langle e_t^{(\varepsilon)}, \nabla \bar{V}_2(Y_t^{(\varepsilon)}) \rangle_{\R^d} \langle e_t^{(\varepsilon)}, dW_t \rangle_{\R^d}. 
	\end{align*}
	As a result, there exists a martingale $M^{(\varepsilon)}$ such that
	\begin{align}
		&d \rho_2(X_t, Y_t^{(\varepsilon)}) \notag \\
		&\quad\leq - \min \left\{ \frac{\zeta}{\beta}, \frac{\lambda}{2}, 2 C \lambda \kappa \right\} \rho_2(X_t, Y_t^{(\varepsilon)}) dt 
		+ U(X_t, Y_t^{(\varepsilon)}) \| \nabla F(Y_t^{(\varepsilon)}) - \nabla G(Y_t^{(\varepsilon)}) \|_{\R^d} dt \notag \\
		&\qquad+ M U(X_t, Y_t^{(\varepsilon)}) \| Z_t^{(\varepsilon)} \|_{\R^d} \{1 -  h_\varepsilon( \| Z_t^{(\varepsilon)} \|_{\R^d} )^2\} dt 
		+ \frac{4 \kappa}{\beta} \{ 1 - h_\varepsilon( \| Z_t^{(\varepsilon)} \|_{\R^d} ) \} h_\varepsilon(\| Z_t^{(\varepsilon)} \|_{\R^d}) \langle e_t^{(\varepsilon)}, \nabla \bar{V}_2(X_t) + \nabla \bar{V}_2(Y_t^{(\varepsilon)}) \rangle_{\R^d} dt \notag \\
		&\qquad + d M_t^{(\varepsilon)}. \label{Ebe_Step3}
	\end{align}
	Since $\| Z_t^{(\varepsilon)} \|_{\R^d}\{ 1 -  h_\varepsilon( \| Z_t^{(\varepsilon)} \|_{\R^d} )^2 \} \leq 2 \varepsilon$, taking expectation in both sides of (\ref{Ebe_Step3}) and the limit $\varepsilon \to 0$, Lemma \ref{Vanish_Occupation_Time} and Theorem \ref{Main_Theorem_Weak_Conv} yield
	\begin{align*}
		\mathcal{W}_{\rho_2}(X_t, Y_t)
		\leq E[\rho_2(X_0, Y_0)]
		- c \int_0^t \mathcal{W}_{\rho_2}(X_s, Y_s) ds + \int_0^t E[U(X_s, Y_s) \| \nabla F(Y_s) - \nabla G(Y_s) \|_{\R^d}] ds.
	\end{align*}
	Thus, $d (e^{c t} \mathcal{W}_{\rho_2}(X_t, Y_t) ) \leq e^{c t} E[U(X_t, Y_t) \| \nabla F(Y_t) - \nabla G(Y_t) \|_{\R^d}] dt$, and therefore we obtain
	\begin{align*}
		\mathcal{W}_{\rho_2}(X_t, Y_t)
		&\leq e^{- c t} E[\rho_2(X_0, Y_0)] 
		+ e^{- c t} \int_0^t e^{c s} E[U(X_s, Y_s) \| \nabla F(Y_s) - \nabla G(Y_s) \|_{\R^d}] ds \\
		&\leq e^{- c t} E[\rho_2(X_0, Y_0)] 
		+ \sup_{u \geq 0} E[U(X_u, Y_u)^2]^{1/2} e^{- c t} \int_0^t e^{c s} E[\| \nabla F(Y_s) - \nabla G(Y_s) \|_{\R^d}^2]^{1/2} ds.
	\end{align*}
	Since $\sup_{u \geq 0} E[U(X_u, Y_u)^2]^{1/2} = O_\alpha(1)$ by Lemma \ref{Lp_bound_SGLD}, the proof is completed. 
\end{proof}


As a corollary to Proposition \ref{Main_Theorem_GenBound}, we can bound the difference between Gibbs measures $\pi_{\beta, F}(dx) \propto e^{- \beta F(x)} dx$ and $\pi_{\beta, G}(dx) \propto e^{- \beta G(x)} dx$ in term of the difference between gradients $\nabla F$ and $\nabla G$. 
The integral of a function $h$ with respect to measure $\mu$ is denoted by $\mu(h)$. 


\begin{cor}
	\label{Corr_GenBound}
	Let $F$ and $G$ be nonnegative and let $H \in C^1(\R^d; \R)$ be $M^\prime$-smooth. 
	Then, we have 
	\begin{align}
		\label{Corr_Result_GenBound}
		| \pi_{\beta, F}(H) - \pi_{\beta, G}(H) | 
		\leq O_{m, b, M, M^\prime, \beta, \| \nabla H(0) \|_{\R^d}, \| \nabla F(0) \|_{\R^d}, \| \nabla G(0) \|_{\R^d}, d} \left( \left( \int_{\R^d} \| \nabla F(x) - \nabla G(x) \|_{\R^d}^2 \pi_{\beta, G}(dx) \right)^{1/2} \right).
	\end{align}
\end{cor}


\begin{proof}
	According to (5.4.58) in \citep{Ikewata}, Gibbs measures $\pi_{\beta, F}$ and $\pi_{\beta, G}$ are invariant measures for $X$ and $Y$ defined by (\ref{LSDE_along_F}) and (\ref{LSDE_along_G}) , respectively. 
	Thus, applying Proposition \ref{Main_Theorem_GenBound} for $\mathcal{L}(X_0) = \pi_{\beta, F}$ and $\mathcal{L}(Y_0) = \pi_{\beta, G}$ and taking the limit $t \to \infty$, we obtain by Lemma \ref{WeakConv_TimeAve} that 
	\begin{align*}
		\mathcal{W}_{\rho_2}(\pi_{\beta, F}, \pi_{\beta, G}) 
		\leq O_{m, b, M, \beta, \| \nabla H(0) \|_{\R^d}, \| \nabla F(0) \|_{\R^d}, \| \nabla G(0) \|_{\R^d}, d} \left( \left( \int_{\R^d} \| \nabla F(x) - \nabla G(x) \|_{\R^d}^2 \pi_{\beta, G}(dx) \right)^{1/2} \right).
	\end{align*}
	Since we have by Lemma \ref{Bound_By_Rho}
	\[
	| \pi_{\beta, F}(H) - \pi_{\beta, G}(H) | 
	\leq O_{m, b, M^\prime, \beta, \| \nabla H(0) \|_{\R^d}, d} \left( \mathcal{W}_{\rho_2}(\pi_{\beta, F}, \pi_{\beta, G}) \right), 
	\]
	the result follows from Proposition \ref{Main_Theorem_GenBound}. 
\end{proof}

\subsection{Proof of Theorem \ref{Main_Theorem_SGLD} (1)}

For $z_1, \dots, z_n$ and $z_1^\prime, \dots, z_n^\prime$, suppose that there exists at most one $i_0$ such that $z_{i_0} \neq z_{i_0}^\prime$. 
Then, setting $L_n^\prime(w) = \frac{1}{n} \sum_{i=1}^n \ell(w; z_i^\prime)$ and denoting $X^{(n)}$ by $Y^{(n)}$ when the $L_n$ is replaced by $L_n^{\prime}$ in (\ref{LSDE_along_Ln}), Proposition \ref{Main_Theorem_GenBound} and Lemma \ref{Bound_By_Rho} yield the following for all $z \in \mathcal{Z}$. 
\begin{align*}
	| E[\ell(X_t^{(n)}; z)] - E[\ell(Y_t^{(n)}; z)] | 
	&\leq O_{\alpha_0} \left( e^{-c t}  \int_0^t e^{c s} E[\| \nabla L_n(Y_s^{(n)}) - \nabla L_n^\prime(Y_s^{(n)}) \|_{\R^d}^2]^{1/2} ds \right) \\
	&\leq O_{\alpha_0} \left( n^{-1} \right). 
\end{align*}
Here, we used Lemma \ref{Lp_bound_SGLD} in the second inequality. 
In particular, by Theorem 2.2 in \citep{Hadt}, if $z_1, \dots, z_n$ are IIDs generated from $\mathcal{D}$, then 
\begin{align*}
	| E[L(X_t^{(n)})] - E[L_n(X_t^{(n)})] | 
	\leq O_{\alpha_0} \left( n^{-1} \right)
\end{align*}
holds. 
\qed

\section{Bounds on the Discretization Error for Langevin Dynamics}
\label{SEC_Second_Appli}

In this section, we prove Theorem \ref{Main_Theorem_SGLD} (2) as the second application of Theorem \ref{Main_Theorem_Weak_Conv}. 
In our proof, it is important that Theorem \ref{Main_Theorem_Weak_Conv} admits the case of $G$ is a functional. 

In the following, we fix $\eta > 0$ and assume $F \in C^1(\R^d; \R)$ is $(m, b)$-dissipative and $M$-smooth. 
For initial values $X_0$ and $Y_0$, we consider the solutions $X$ and $Y$ of (\ref{LSDE_along_F}) and 
\begin{align}
	\label{SGLD_along_F}
	dY_t 
	= - \nabla F(Y_{\lfloor t / \eta \rfloor \eta}) dt + \sqrt{2 / \beta} d W_t,
\end{align}
respectively. 
Here, $\lfloor \cdot \rfloor$ denotes the floor function. 
Then we define the ARC between $X$ and $Y$ by
\begin{align}
	\label{StepSize_Approximate_Reflection_Coupling}
	\begin{cases}
		dY_t^{(\varepsilon)} 
		= - \nabla F(Y_{\lfloor t / \eta \rfloor  \eta}^{(\varepsilon)}) dt + \sqrt{2 / \beta} (I_d - 2 h_\varepsilon( \| X_t - Y_t^{(\varepsilon)} \|_{\R^d} ) e_t^{(\varepsilon)} {e_t^{(\varepsilon)}}^\top) d W_t, \\
		Y_0^{(\varepsilon)} = Y_0.
	\end{cases}
\end{align}
Note that the functional $G$ defined by $G(t, \varphi) = \nabla F(\varphi(\lfloor t / \eta \rfloor \eta))$ for $\varphi \in C([0, \infty); \R^d)$ is progressively measurable and satisfies  (\ref{Lipschitz_Functional}) and (\ref{Linear_Growth_Functional}).

\subsection{Uniform bound on the time for the difference between (\ref{LSDE_along_F}) and (\ref{SGLD_along_F})}

Proposition \ref{Main_Theorem_StepSizeBound} below gives a uniform evaluation on the time to the discretization error of (\ref{LSDE_along_F}) with order $\eta^{1/2}$. 
Since the proof of Proposition \ref{Main_Theorem_StepSizeBound} is quite similar to that of Proposition \ref{Main_Theorem_GenBound}, we only explain the important difference of those in the following proof. 


\begin{prop}
	\label{Main_Theorem_StepSizeBound}
	Let $X$ and $Y$ be the solutions of (\ref{LSDE_along_F}) and (\ref{SGLD_along_F}) with initial values $X_0, Y_0 \in L^4(\Omega; \R^d)$, respectively. 
	Then there exists some $\eta_0 = O_{m, M}(1)$ and the following inequality holds uniformly on $0 < \eta \leq \eta_0$.
	\begin{align}
		\label{Main_Result_StepSizeBound}
		\mathcal{W}_{\rho_2}(X_t, Y_t)
		&\leq e^{- c t} E[\rho_2(X_0, Y_0)] 
		+ O_{m, b, M, \beta, \|F(0)\|_{\R^d}, E[\| X_0 \|_{\R^d}^4], E[\| Y_0 \|_{\R^d}^4], d}(\eta^{1/2}).
	\end{align}
	Here, $c > 0$ is the constant defined in Proposition \ref{Main_Theorem_GenBound}.  
\end{prop}


\begin{proof}
	Fix $\varepsilon > 0$ and let $Z_t^{(\varepsilon)} = X_t - Y_t^{(\varepsilon)}$ for $X_t$ and $Y_t^{(\varepsilon)}$ defined by (\ref{StepSize_Approximate_Reflection_Coupling}). \\
	
	
	\noindent {\bf \underline{Step 1} }\ Evaluation of $f(\| Z_t^{(\varepsilon)} \|_{\R^d})$. \\
	
	\noindent In this case, (\ref{M_smooth_Bound}) is replaced by
	\begin{align*}
		&\langle e_t^{(\varepsilon)}, \nabla F(X_t) - \nabla F(Y_{\lfloor s / \eta \rfloor  \eta}^{(\varepsilon)}) \rangle_{\R^d} 
		\leq M \| Y_t^{(\varepsilon)} - Y_{\lfloor t / \eta \rfloor  \eta}^{(\varepsilon)} \|_{\R^d}  
		+ M \| Z_t^{(\varepsilon)} \|_{\R^d}.
	\end{align*}
	Thus, $f(\| Z_t^{(\varepsilon)} \|_{\R^d})$ is a semi-martingale that satisfies
	\begin{align*}
		d f(\| Z_t^{(\varepsilon)} \|_{\R^d}) 
		&\leq M \| Y_t^{(\varepsilon)} - Y_{\lfloor t / \eta \rfloor  \eta}^{(\varepsilon)} \|_{\R^d} dt 
		+ M \| Z_t^{(\varepsilon)} \|_{\R^d} \{ 1 -  h_\varepsilon( \| Z_t^{(\varepsilon)} \|_{\R^d} )^2\} dt \\
		&\quad+ 2 \sqrt{\frac{2}{\beta}} f^\prime(\| Z_t^{(\varepsilon)} \|_{\R^d}) h_\varepsilon(\| Z_t^{(\varepsilon)} \|_{\R^d}) \langle e_t^{(\varepsilon)}, dW_t \rangle_{\R^d} 
		- \frac{8Q(\kappa)}{\beta} f^\prime(\| Z_t^{(\varepsilon)} \|_{\R^d}) h_\varepsilon( \| Z_t^{(\varepsilon)} \|_{\R^d} )^2 dt \\
		&\quad- \frac{\zeta}{\beta} f(\| Z_t^{(\varepsilon)} \|_{\R^d}) \chi_{(0, R_2)}(\| Z_t^{(\varepsilon)} \|_{\R^d}) dt
		- \frac{\xi}{\beta} f(\| Z_t^{(\varepsilon)} \|_{\R^d}) \chi_{(0, R_1)}(\| Z_t^{(\varepsilon)} \|_{\R^d}) dt.
	\end{align*}
	
	
	\noindent {\bf \underline{Step 2} }\ Evaluation of $U(X_t, Y_t^{(\varepsilon)})$. \\
	
	\noindent In this case, the term $2 \kappa \langle Y_t^{(\varepsilon)}, \nabla F(Y_{\lfloor t / \eta \rfloor  \eta}^{(\varepsilon)}) - \nabla F(Y_t^{(\varepsilon)}) \rangle_{\R^d} dt$ is added to (\ref{Ito_Formula_For_U}). 
	Thus, we have
	\begin{align*}
		d U(X_t, Y_t^{(\varepsilon)}) 
		&\leq \left( \frac{\xi}{\beta} U(X_t, Y_t^{(\varepsilon)}) \chi_{(0, R_1]}(\| Z_t^{(\varepsilon)} \|_{\R^d}) -\frac{\lambda}{2} \min\{ 1, 4 C \kappa \} U(X_t, Y_t^{(\varepsilon)}) \chi_{(R_2, \infty)}(\| Z_t^{(\varepsilon)} \|_{\R^d}) \right) dt \\
		&\quad+2 M \kappa \| Y_t^{(\varepsilon)} \|_{\R^d} M \| Y_t^{(\varepsilon)} - Y_{\lfloor t / \eta \rfloor  \eta}^{(\varepsilon)} \|_{\R^d} dt \\
		&\quad+ \kappa \sqrt{\frac{2}{\beta}} \langle \nabla \bar{V}_2(X_t) + \nabla \bar{V}_2(Y_t^{(\varepsilon)}), dW_t \rangle_{\R^d} 
		-2 \kappa \sqrt{\frac{2}{\beta}} h_\varepsilon( \| Z_t^{(\varepsilon)} \|_{\R^d} ) \langle e_t^{(\varepsilon)}, \nabla \bar{V}_2(Y_t^{(\varepsilon)}) \rangle_{\R^d} \langle e_t^{(\varepsilon)}, dW_t \rangle_{\R^d}.
	\end{align*}
	
	
	\noindent {\bf \underline{Step 3} }\ Evaluation of $\rho_2(X_t, Y_t^{(\varepsilon)})$. \\
	
	\noindent The martingale parts of semi-martingales $f(\| Z_t^{(\varepsilon)} \|_{\R^d})$ and $U(X_t, Y_t^{(\varepsilon)})$ are the same as those in the proof of Proposition \ref{Main_Theorem_GenBound}. 
	Therefore, (\ref{Bound_Quadratic_Variation}) holds. 
	For (\ref{Ito_Formula_Rho}), in this case, $M \| Y_t^{(\varepsilon)} - Y_{\lfloor t / \eta \rfloor  \eta}^{(\varepsilon)} \|_{\R^d} dt$ and 
	$2 M \kappa \| Y_t^{(\varepsilon)} \|_{\R^d} \| Y_t^{(\varepsilon)} - Y_{\lfloor t / \eta \rfloor  \eta}^{(\varepsilon)} \|_{\R^d} dt$ are added to the upper bounds of $d f(\| Z_t^{(\varepsilon)} \|_{\R^d})$ and $d U(X_t, Y_t^{(\varepsilon)})$, respectively. 
	As a result, there exists a martingale $M^{(\varepsilon)}$ such that
	\begin{align*}
		&d \rho_2(X_t, Y_t^{(\varepsilon)}) \\
		&\quad\leq - \min \left\{ \frac{\zeta}{\beta}, \frac{\lambda}{2}, 2 C \lambda \kappa \right\} \rho_2(X_t, Y_t^{(\varepsilon)}) dt 
		+ M \{ U(X_t, Y_t^{(\varepsilon)}) + 2 \kappa f(\| Z_t^{(\varepsilon)} \|_{\R^d}) \| Y_t^{(\varepsilon)} \|_{\R^d} \} \| Y_t^{(\varepsilon)} - Y_{\lfloor t / \eta \rfloor  \eta}^{(\varepsilon)} \|_{\R^d} dt \\
		&\qquad+ M U(X_t, Y_t^{(\varepsilon)}) \| Z_t^{(\varepsilon)} \|_{\R^d} \{1 -  h_\varepsilon( \| Z_t^{(\varepsilon)} \|_{\R^d} )^2\} dt 
		+ \frac{4 \kappa}{\beta} \{ 1 - h_\varepsilon( \| Z_t^{(\varepsilon)} \|_{\R^d} ) \} h_\varepsilon(\| Z_t^{(\varepsilon)} \|_{\R^d}) \langle e_t^{(\varepsilon)}, \nabla \bar{V}_2(X_t) + \nabla \bar{V}_2(Y_t^{(\varepsilon)}) \rangle_{\R^d} dt \\
		&\qquad + d M_t^{(\varepsilon)}.
	\end{align*}
	Hence, as in the proof of Proposition \ref{Main_Theorem_GenBound}, we obtain
	\begin{align*}
		\mathcal{W}_{\rho_2}(X_t, Y_t)
		\leq E[\rho_2(X_0, Y_0)] 
		- c \int_0^t \mathcal{W}_{\rho_2}(X_s, Y_s) ds 
		+ M \int_0^t E[\{ U(X_s, Y_s) + 2 \kappa f(R_2) \| Y_s \|_{\R^d} \}\| Y_s - Y_{\lfloor s / \eta \rfloor  \eta} \|_{\R^d}] ds. 
	\end{align*}
	Thus, since
	\begin{align*}
		d (e^{c t} \mathcal{W}_{\rho_2}(X_t, Y_t) ) 
		\leq M e^{c t} E[\{ U(X_t, Y_t) + 2 \kappa f(R_2) \| Y_t \|_{\R^d} \}\| Y_t - Y_{\lfloor t / \eta \rfloor  \eta} \|_{\R^d}] dt,
	\end{align*}
	we obtain 
	\begin{align*}
		\mathcal{W}_{\rho_2}(X_t, Y_t)
		&\leq e^{- c t} E[\rho_2(X_0, Y_0)]
		+ M e^{- c t} \int_0^t e^{c s} E[\{ U(X_s, Y_s) + 2 \kappa f(R_2) \| Y_s \|_{\R^d} \} \| Y_s - Y_{\lfloor s / \eta \rfloor  \eta} \|_{\R^d}] ds \\
		&\leq e^{- c t} E[\rho_2(X_0, Y_0)]
		+ M \sup_{u \geq 0} E[\{ U(X_s, Y_s) + 2 \kappa f(R_2) \| Y_s \|_{\R^d} \}^2]^{1/2} e^{- c t} \int_0^t e^{c s} E[\| Y_s - Y_{\lfloor s / \eta \rfloor  \eta} \|_{\R^d}^2]^{1/2} ds. 
	\end{align*}
	Lemmas \ref{Lp_bound_SGD} and \ref{OneStepBound} complete the proof. 
\end{proof}

\subsection{Proof of Theorem \ref{Main_Theorem_SGLD} (2)}

By Lemma \ref{Bound_By_Rho}, we have
\begin{align*}
	| E[L(X_t^{(n, \eta)})] - E[L(X_t^{(n)})] | 
	\leq O_{m, b, M, \beta, A, d} (\mathcal{W}_{\rho_2}(X_t^{(n, \eta)}, X_t^{(n)}) ).
\end{align*}
Thus, applying Proposition \ref{Main_Theorem_StepSizeBound} to $F = L_n$, we obtain the desired result. 
\qed

\section{Bounds on the Mini-batch Error for Stochastic Gradient Langevin Dynamics}
\label{SEC_Third_Appli}

In this section, we prove Theorem \ref{Main_Theorem_SGLD} (3) as the third application of Theorem \ref{Main_Theorem_Weak_Conv}. 
As in the previous section, it is important that Theorem \ref{Main_Theorem_Weak_Conv} admits the case of $G$ is a functional. 

In the following, we fix $\eta > 0$ and for simplification, abbreviate $X^{(n, \eta, B)}$ and $X^{(n, \eta)}$ defined by (\ref{SGLD_along_Ln_Batch}) and (\ref{SGLD_along_Ln}) as $X$ and $Y$, respectively.
Then, we define the ARC between $X$ and $Y$ by
\begin{align}
	\label{BatchSize_Approximate_Reflection_Coupling}
	\begin{cases}
		dY_t^{(\varepsilon)} 
		= - \nabla L_n(Y_{\lfloor t / \eta \rfloor  \eta}^{(\varepsilon)}) dt + \sqrt{2 / \beta} (I_d - 2 h_\varepsilon( \| X_t - Y_t^{(\varepsilon)} \|_{\R^d} ) e_t^{(\varepsilon)} {e_t^{(\varepsilon)}}^\top) d W_t, \\
		Y_t^{(\varepsilon)} = X_0. 
	\end{cases}
\end{align}

\subsection{Uniform bound on the time for the difference between (\ref{SGLD_along_Ln_Batch}) and (\ref{SGLD_along_Ln})}

Proposition \ref{Main_Theorem_BatchSizeBound} below gives a uniform evaluation on the time to the mini batch error of (\ref{SGLD_along_Ln}) with order $\sqrt{(n - B)/ B (n-1)}$. 
As in the previous section, we only explain the important difference between the proof of the following Proposition \ref{Main_Theorem_BatchSizeBound} and that of Proposition \ref{Main_Theorem_GenBound} in the following proof. 


\begin{prop}
	\label{Main_Theorem_BatchSizeBound}
	Under Assumption \ref{Assum_Main2}, there exists some $\eta_0 = O_{m, M}(1)$ and the following inequality holds uniformly on $0 < \eta \leq \eta_0$.
	\begin{align}
		\mathcal{W}_{\rho_2}(X_t, Y_t)
		&\leq O_{\alpha_0} ( \eta^{1/2} + \sqrt{(n - B) / B (n - 1)} ).
	\end{align}
\end{prop}


\begin{proof}
	Fix $\varepsilon > 0$ and let $Z_t^{(\varepsilon)} = X_t - Y_t^{(\varepsilon)}$ for $X_t$ and $Y_t^{(\varepsilon)}$ defined by (\ref{BatchSize_Approximate_Reflection_Coupling}). \\
	
	
	\noindent {\bf \underline{Step 1} }\ Evaluation of $f(\| Z_t^{(\varepsilon)} \|_{\R^d})$. \\
	
	\noindent In this case, (\ref{M_smooth_Bound}) is replaced by 
	\begin{align*}
		&\langle e_t^{(\varepsilon)}, \nabla L_{n, \lfloor t / \eta \rfloor}(X_{\lfloor t / \eta \rfloor  \eta}) - \nabla L_n(Y_{\lfloor t / \eta \rfloor  \eta}^{(\varepsilon)}) \rangle_{\R^d} \\
		&\quad\leq M \| X_t - X_{\lfloor t / \eta \rfloor  \eta} \|_{\R^d}
		+ M \| Z_t^{(\varepsilon)} \|_{\R^d} 
		+ \| \nabla L_n(Y_t^{(\varepsilon)}) - \nabla L_{n, \lfloor t / \eta \rfloor}(Y_t^{(\varepsilon)}) \|_{\R^d}
		+ M \| Y_t^{(\varepsilon)} - Y_{\lfloor t / \eta \rfloor  \eta}^{(\varepsilon)} \|_{\R^d}.
	\end{align*}
	Thus, $f(\| Z_t^{(\varepsilon)} \|_{\R^d})$ is a semi-martingale that satisfies
	\begin{align*}
		d f(\| Z_t^{(\varepsilon)} \|_{\R^d}) 
		&\leq \| \nabla L_n(Y_t^{(\varepsilon)}) - \nabla L_{n, \lfloor t / \eta \rfloor}(Y_t^{(\varepsilon)}) \|_{\R^d} dt 
		+ M \| Z_t^{(\varepsilon)} \|_{\R^d} \{ 1 -  h_\varepsilon( \| Z_t^{(\varepsilon)} \|_{\R^d} )^2 \} dt \\
		&\quad+ M \| Y_t^{(\varepsilon)} - Y_{\lfloor t / \eta \rfloor  \eta}^{(\varepsilon)} \|_{\R^d} dt 
		+ M \| X_t - X_{\lfloor t / \eta \rfloor  \eta} \|_{\R^d} dt \\
		&\quad+ 2 \sqrt{\frac{2}{\beta}} f^\prime(\| Z_t^{(\varepsilon)} \|_{\R^d}) h_\varepsilon(\| Z_t^{(\varepsilon)} \|_{\R^d}) \langle e_t^{(\varepsilon)}, dW_t \rangle_{\R^d} 
		- \frac{8Q(\kappa)}{\beta} f^\prime(\| Z_t^{(\varepsilon)} \|_{\R^d}) h_\varepsilon( \| Z_t^{(\varepsilon)} \|_{\R^d} )^2 dt \\
		&\quad- \frac{\zeta}{\beta} f(\| Z_t^{(\varepsilon)} \|_{\R^d}) \chi_{(0, R_2)}(\| Z_t^{(\varepsilon)} \|_{\R^d}) dt
		- \frac{\xi}{\beta} f(\| Z_t^{(\varepsilon)} \|_{\R^d}) \chi_{(0, R_1)}(\| Z_t^{(\varepsilon)} \|_{\R^d}) dt. 
	\end{align*}
	
	
	\noindent {\bf \underline{Step 2} }\ Evaluation of $U(X_t, Y_t^{(\varepsilon)})$. \\
	
	\noindent In this case, the terms 
	$2 \kappa \langle Y_t^{(\varepsilon)}, \nabla L_n(Y_{\lfloor t / \eta \rfloor  \eta}^{(\varepsilon)}) - \nabla L_n(Y_t^{(\varepsilon)}) \rangle_{\R^d} dt$
	and
	$2 \kappa \langle X_t, \nabla L_{n, \lfloor t / \eta \rfloor}(X_{\lfloor t / \eta \rfloor  \eta}) - \nabla L_{n, \lfloor t / \eta \rfloor}(X_t) \rangle_{\R^d} dt$
	are added to (\ref{Ito_Formula_For_U}).
	Thus, we have
	\begin{align*}
		d U(X_t, Y_t^{(\varepsilon)}) 
		&\leq \left( \frac{\xi}{\beta} U(X_t, Y_t^{(\varepsilon)}) \chi_{(0, R_1]}(\| Z_t^{(\varepsilon)} \|_{\R^d}) -\frac{\lambda}{2} \min\{ 1, 4 C \kappa \} U(X_t, Y_t^{(\varepsilon)}) \chi_{(R_2, \infty)}(\| Z_t^{(\varepsilon)} \|_{\R^d}) \right) dt \\
		&\quad+2 M \kappa \| Y_t^{(\varepsilon)} \|_{\R^d} \| Y_t^{(\varepsilon)} - Y_{\lfloor t / \eta \rfloor  \eta}^{(\varepsilon)} \|_{\R^d} dt 
		+2 M \kappa \| X_t \|_{\R^d} \| X_t - X_{\lfloor t / \eta \rfloor  \eta} \|_{\R^d} dt \\
		&\quad+ \kappa \sqrt{\frac{2}{\beta}} \langle \nabla \bar{V}_2(X_t) + \nabla \bar{V}_2(Y_t^{(\varepsilon)}), dW_t \rangle_{\R^d} 
		-2 \kappa \sqrt{\frac{2}{\beta}} h_\varepsilon( \| Z_t^{(\varepsilon)} \|_{\R^d} ) \langle e_t^{(\varepsilon)}, \nabla \bar{V}_2(Y_t^{(\varepsilon)}) \rangle_{\R^d} \langle e_t^{(\varepsilon)}, dW_t \rangle_{\R^d}.
	\end{align*}
	
	
	\noindent {\bf \underline{Step 3} }\ Evaluation of $\rho_2(X_t, Y_t^{(\varepsilon)})$. \\
	
	\noindent The martingale parts of semi-martingales $f(\| Z_t^{(\varepsilon)} \|_{\R^d})$ and $U(X_t, Y_t^{(\varepsilon)})$ are the same as those in the proof of Proposition \ref{Main_Theorem_GenBound}. 
	Therefore, (\ref{Bound_Quadratic_Variation}) holds. 
	For (\ref{Ito_Formula_Rho}), in this case, 
	\begin{align*}
		\| \nabla L_n(Y_t^{(\varepsilon)}) - \nabla L_{n, \lfloor t / \eta \rfloor}(Y_t^{(\varepsilon)}) \|_{\R^d} dt 
		+ M \| Y_t^{(\varepsilon)} - Y_{\lfloor t / \eta \rfloor  \eta}^{(\varepsilon)} \|_{\R^d} dt 
		+ M \| X_t - X_{\lfloor t / \eta \rfloor  \eta} \|_{\R^d} dt
	\end{align*}
	and 
	$2 M \kappa \| Y_t^{(\varepsilon)} \|_{\R^d} \| Y_t^{(\varepsilon)} - Y_{\lfloor t / \eta \rfloor  \eta}^{(\varepsilon)} \|_{\R^d} dt 
	+2 M \kappa \| X_t \|_{\R^d} \| X_t - X_{\lfloor t / \eta \rfloor  \eta} \|_{\R^d} dt$
	are added to the upper bounds of $d f(\| Z_t^{(\varepsilon)} \|_{\R^d})$ and $d U(X_t, Y_t^{(\varepsilon)})$, respectively. 
	As a result, there exists a martingale $M^{(\varepsilon)}$ such that
	\begin{align*}
		&d \rho_2(X_t, Y_t^{(\varepsilon)}) \\
		&\quad\leq - \min \left\{ \frac{\zeta}{\beta}, \frac{\lambda}{2}, 2 C \lambda \kappa \right\} \rho_2(X_t, Y_t^{(\varepsilon)}) dt 
		+ U(X_t, Y_t^{(\varepsilon)}) \| \nabla L_n(Y_t^{(\varepsilon)}) - \nabla L_{n, \lfloor t / \eta \rfloor}(Y_t^{(\varepsilon)}) \|_{\R^d} dt \\
		&\qquad+ M \{ U(X_t, Y_t^{(\varepsilon)}) + 2 \kappa f(\| Z_t^{(\varepsilon)} \|_{\R^d}) \| Y_t^{(\varepsilon)} \|_{\R^d} \} \| Y_t^{(\varepsilon)} - Y_{\lfloor t / \eta \rfloor  \eta}^{(\varepsilon)} \|_{\R^d} dt \\
		&\qquad+ M \{ U(X_t, Y_t^{(\varepsilon)}) + 2 \kappa f(\| Z_t^{(\varepsilon)} \|_{\R^d}) \| X_t \|_{\R^d} \} \| X_t - X_{\lfloor t / \eta \rfloor  \eta} \|_{\R^d} dt  \\
		&\qquad+ M U(X_t, Y_t^{(\varepsilon)}) \| Z_t^{(\varepsilon)} \|_{\R^d} \{1 -  h_\varepsilon( \| Z_t^{(\varepsilon)} \|_{\R^d} )^2\} dt 
		+ \frac{4 \kappa}{\beta} \{ 1 - h_\varepsilon( \| Z_t^{(\varepsilon)} \|_{\R^d} ) \} h_\varepsilon(\| Z_t^{(\varepsilon)} \|_{\R^d}) \langle e_t^{(\varepsilon)}, \nabla \bar{V}_2(X_t) + \nabla \bar{V}_2(Y_t^{(\varepsilon)}) \rangle_{\R^d} dt \\
		&\qquad + d M_t^{(\varepsilon)}. 
	\end{align*}
	Hence, as in the proof of Proposition \ref{Main_Theorem_GenBound}, we obtain
	\begin{align*}
		\mathcal{W}_{\rho_2}(X_t, Y_t)
		&\leq - c \int_0^t \mathcal{W}_{\rho_2}(X_s, Y_s) ds 
		+ \int_0^t E[U(X_s, Y_s) \| \nabla L_n(Y_{\lfloor s / \eta \rfloor  \eta}) - \nabla L_{n, \lfloor s / \eta \rfloor}(Y_{\lfloor s / \eta \rfloor  \eta}) \|_{\R^d}] ds \\
		&\quad + M \int_0^t E[\{ U(X_s, Y_s) + 2 \kappa f(R_2) (\| Y_s \|_{\R^d} + \| X_s \|_{\R^d}) \} (\| Y_s - Y_{\lfloor s / \eta \rfloor  \eta} \|_{\R^d} 
		+ \| X_s - X_{\lfloor s / \eta \rfloor \eta} \|_{\R^d} )] ds.
	\end{align*}
	Thus, since  
	\begin{align*}
		d (e^{c t} \mathcal{W}_{\rho_2}(X_t, Y_t) ) 
		&\leq e^{c t} E[U(X_t, Y_t) \| \nabla L_n(Y_{\lfloor t / \eta \rfloor \eta}) - \nabla L_{n, \lfloor t / \eta \rfloor}(Y_{\lfloor t / \eta \rfloor  \eta}) \|_{\R^d}] dt \\
		&\quad+ M e^{c t} E[\{ U(X_t, Y_t) + 2 \kappa f(R_2) (\| Y_t \|_{\R^d} + \| X_t \|_{\R^d}) \} (\| Y_t - Y_{\lfloor t / \eta \rfloor \eta} \|_{\R^d} 
		+ \| X_t - X_{\lfloor t / \eta \rfloor \eta} \|_{\R^d} )] dt
	\end{align*}
	and $M \sup_{u \geq 0} E[\{ U(X_u, Y_u) + 2 \kappa f(R_2) (\| Y_u \|_{\R^d} + \| X_u \|_{\R^d}) \}^2]^{1/2} = O_{\alpha_0} (1)$ holds by Lemma \ref{Lp_bound_SGD}, we obtain 
	\begin{align*}
		\mathcal{W}_{\rho_2}(X_t, Y_t) 
		&\leq O_{\alpha_0} \left( e^{-c t} \int_0^t e^{ c s} E[\| \nabla L_n(Y_{\lfloor s / \eta \rfloor  \eta}) - \nabla L_{n, \lfloor s / \eta \rfloor}(Y_{\lfloor s / \eta \rfloor  \eta}) \|_{\R^d}^2]^{1/2} ds \right. \\
		&\quad+ \left. e^{-c t} \int_0^t e^{ c s} \{ E[\| Y_s - Y_{\lfloor s / \eta \rfloor  \eta} \|_{\R^d}^2]^{1/2} + E[\| X_s - X_{\lfloor s / \eta \rfloor \eta} \|_{\R^d}^2]^{1/2} \} ds \right). 
	\end{align*}
	Lemmas \ref{OneStepBound} and \ref{OneStepBatchBound} complete the proof. 
\end{proof}

\subsection{Proof of Theorem \ref{Main_Theorem_SGLD} (3)} 

By Lemma \ref{Bound_By_Rho}, we have
\begin{align*}
	| E[L(X_t^{(n, \eta, B)})] - E[L(X_t^{(n, \eta)})] | 
	\leq O_{m, b, M, \beta, A, d} (\mathcal{W}_{\rho_2}(X_t^{(n, \eta)}, X_t^{(n, \eta, B)}) ). 
\end{align*}
Thus, Proposition \ref{Main_Theorem_BatchSizeBound} proves Theorem \ref{Main_Theorem_SGLD} (3). 
\qed

\appendix
\renewcommand{\thesection}{\Alph{section}}
\section{Appendix}

\subsection{Lemmas on weak convergence}

\begin{lem}
	\label{WeakConv_TimeAve}
	Let a continuous function $H : \R^d \to \R$ satisfy $|H(x)| \leq K(1 + \| x \|_{\R^d})^p$ for $p > 0$ and $K > 0$. 
	If a $\R^d$-valued process $Y_t$ satisfies $\sup_{t \geq 0} E[\| Y_t \|_{\R^d}^q] < \infty$ for some $q > p$ and converges weakly to $Y$ as $t \to\infty$, then
	\[
	\lim_{t \to \infty} c e^{- ct} \int_0^t e^{c s} E[H(Y_s)] ds 
	= E[H(Y)]
	\]
	holds for all $c > 0$. 
\end{lem}


\begin{proof}
	Fix an arbitrary $\varepsilon > 0$. 
	By the assumption, we can take $R > 0$ so that 
	\[
	E[|H(Y_t)|; \| Y_t \|_{\R^d} \geq R] 
	\leq \frac{K E[(1 + \| Y_t \|_{\R^d})^q]}{R^{q-p}} 
	\leq \varepsilon
	\]
	holds. 
	For this $R$, if we define
	\[
	\phi_R(a)
	= \begin{cases}
		1, & a \leq R, \\
		1 - (a - R), & R \leq a \leq R+1, \\
		0, & a \geq R+1
	\end{cases} 
	\]
	and $H_R(x) = H(x) \phi_R(\| x \|_{\R^d})$, then we have $|H(x) - H_R(x)| = |H(x)| (1 - \phi_R(x)) \leq  |H(x)| \chi_{\{ \| x \|_{\R^d} \geq R \}}$. 
	In particular, $|E[H(Y_t)] - E[H_R(Y_t)]| \leq E[|H(Y_t)|; \| Y_t \|_{\R^d} \geq R] \leq \varepsilon$ holds for all $t \geq 0$. 
	Therefore, we only have to handle the case of $H$ is bounded. 
	
	Fix arbitrarily $\delta > 0$. 
	Since $Y_t$ converges weakly to $Y$ and $H$ is bounded and continuous, we can take $T > 0$ so that $|E[H(Y_t)] - E[H(Y)]| \leq \delta$ holds for all $t \geq T$. 
	Thus, if $t \geq T$, then
	\begin{align*}
		\left| c e^{- ct} \int_0^t e^{c s} E[H(Y_s)] ds - E[H(Y)] \right| 
		&\leq c e^{- ct} \int_0^t e^{c s} | E[H(Y_s)] - E[H(Y)] | ds + e^{- ct} \sup_{x \in \R^d} |H(x)| \\
		& \leq c e^{- ct} \int_0^T e^{c s} |E[H(Y_s)] - E[H(Y)]| ds + c e^{- ct} \int_T^t e^{c s} \delta ds + e^{- ct} \sup_{x \in \R^d} |H(x)| \\
		&\leq 3 e^{- c(t - T)} \sup_{x \in \R^d} |H(x)| + \delta
	\end{align*}
	holds.  
\end{proof}


As in the proof of Lemma \ref{WeakConv_TimeAve}, we can prove the following by reducing the proof to the case of $H(t, \cdot)$ is bounded. 

\begin{lem}
	\label{WeakConv_Non_Bdd}
	For a functional $H : [0, \infty) \times C([0, \infty); \R^d) \to \R$, assume that $H(t, \cdot)$ is continuous and satisfies (\ref{Linear_Growth_Functional}) for all $t \geq 0$. 
	Furthermore, for a sequence of $C([0, \infty); \R^d)$-valued random variables $Z^{(n)}$, assume that for each $t \geq 0$, there exists some $p(t) >1$ such that $\sup_{n \in \N} E[\sup_{0 \leq s \leq t} \| Z_s^{(n)} \|_{\R^d}^{p(t)}] < \infty$. 
	If $Z^{(n)}$ converges weakly to $Z$ as $n \to \infty$, then we have $\lim_{n \to \infty} E[H(t, Z^{(n)})] = E[H(t, Z)]$ for all $t \geq 0$. 
\end{lem}

\subsection{Lemmas on Langevin Dynamics}

\begin{lem}
	\label{p-th_Lyap}
	Let $H \in C^1(\R^d; \R)$ be $(m, b)$-dissipative and let  
	\begin{align}
		\label{Infinitesimal Generator}
		\mathcal{L}_H f(x)
		\coloneqq - \langle \nabla H(x), \nabla f(x) \rangle_{\R^d} + \frac{1}{\beta} \Delta f(x),\qquad f \in C^2(\R^d; \R). 
	\end{align}
	Then, for any $p \geq 2$ and $x \in \R^d$, 
	\[
	\mathcal{L}_H V_p(x)
	\leq C(p) - \lambda(p) V_p(x)
	\]
	holds, where 
	\[
	L(p) 
	= \left\{ \frac{2}{m} \left( \frac{d + p-2}{\beta} + b \right) \right\}^{1/2}
	\]
	and 
	\begin{align}
		\label{p-th_Lyap_Constant}
		\lambda(p) 
		= \frac{mp}{2},\quad 
		C(p) 
		= \lambda(p) L(p)^p.
	\end{align} 
\end{lem}


\begin{proof}
	Since $\nabla V_p(x) = p \| x \|_{\R^d}^{p-2} x$ and $\Delta V_p(x) = p (d + p-2) \| x \|_{\R^d}^{p-2}$, by the $(m, b)$-dissipativity of $H$, we have
	\[
	\langle \nabla H(x), \nabla V_p(x) \rangle_{\R^d} 
	= p \| x \|_{\R^d}^{p-2} \langle \nabla H(x), x \rangle_{\R^d} 
	\geq p \| x \|_{\R^d}^{p-2} (m \|x\|_{\R^d}^2 - b).
	\]
	Thus, we obtain
	\begin{align*}
		- \langle \nabla H(x), \nabla V_p(x) \rangle_{\R^d} + \frac{\Delta V_p(x)}{\beta} 
		&\leq - mp \|x\|_{\R^d}^p + bp \| x \|_{\R^d}^{p-2} + \frac{p (d + p-2)}{\beta} \| x \|_{\R^d}^{p-2} \\
		&= - \frac{mp}{2} \|x\|_{\R^d}^p - \left( \frac{p (d + p-2)}{\beta} + bp - \frac{mp}{2} \|x\|_{\R^d}^p \right) \| x \|_{\R^d}^{p-2}.
	\end{align*}
	In particular, if $\| x \|_{\R^d} \geq L(p)$, then $\mathcal{L}_H V_p(x) \leq - \lambda(p) V_p(x)$ holds. 
	On the other hand, if $\| x \|_{\R^d} \leq L(p)$, then 
	\[
	\mathcal{L}_H V_p(x) 
	\leq - \lambda(p) V_p(x) + \left( \frac{p (d + p-2)}{\beta} + bp \right) L(p)^{p-2}
	= C(p) - \lambda(p) V_p(x)
	\]
	holds. 
\end{proof}


\begin{lem}
	\label{Lp_bound_SGLD}
	Let $H \in C^1(\R^d; \R)$ be $M$-smooth and let $p \geq 2$. 
	For a $\{ \F_t \}$-adapted process $\gamma_t$ and $Y_0 \in L^p(\Omega; \R^d)$, there exists uniquely the strong solution $Y^{(\varepsilon)}$ of 
	\begin{align}
		\label{BASE_ApproxCoupling}
		\begin{cases}
			dY_t^{(\varepsilon)} 
			= - \nabla H(Y_t^{(\varepsilon)}) dt + \sqrt{2 / \beta} (I_d - 2 h_\varepsilon( \| \gamma_t - Y_t^{(\varepsilon)} \|_{\R^d} ) e_t^{(\varepsilon)} {e_t^{(\varepsilon)}}^\top) dW_t, & t \geq 0, \\
			Y_0^{(\varepsilon)} = Y_0.
		\end{cases}
	\end{align}
	Here, $e_t^{(\varepsilon)} = (\gamma_t - Y_t^{(\varepsilon)}) / \| \gamma_t - Y_t^{(\varepsilon)} \|_{\R^d}$. 
	Furthermore, if $H$ is $(m, b)$-dissipative, then, for all $t \geq 0$, $Y^{(\varepsilon)}$ satisfies
	\begin{align}
		\label{Lp_Bound_SGLD}
		E[\| Y_t^{(\varepsilon)} \|_{\R^d}^p] 
		\leq e^{-\lambda(p) t} E[\| Y_0 \|_{\R^d}^p] + \frac{C(p)}{\lambda(p)} (1 - e^{-\lambda(p) t}).
	\end{align}
	The same bound as (\ref{Lp_Bound_SGLD}) holds also for the solution of
	\begin{align}
		\label{LSDE_along_H_new}
		dY_t 
		= - \nabla H(Y_t) dt + \sqrt{2 / \beta} W_t,\qquad t \geq 0.
	\end{align}
\end{lem}


\begin{proof}
	For each $\varepsilon > 0$, the map $\R^d \ni x \mapsto h_\varepsilon(\| x \|_{\R^d}) ( x / \| x \|_{\R^d} ) ( x / \| x \|_{\R^d} )^\top \in \R^d \otimes \R^d$ is Lipschitz continuous. 
	Thus, (\ref{BASE_ApproxCoupling}) has the unique strong solution $Y_t^{(\varepsilon)}$. 
	Denoting $Z_t^{(\varepsilon)} = \gamma_t - Y_t^{(\varepsilon)}$, Ito's formula yields 
	\begin{align*}
		d V_p(Y_t^{(\varepsilon)}) 
		&= \mathcal{L}_H V(Y_t^{(\varepsilon)}) dt 
		- \frac{4 h_\varepsilon(\| Z_t^{(\varepsilon)} \|_{\R^d}) (1 - h_\varepsilon(\| Z_t^{(\varepsilon)} \|_{\R^d}))}{\beta} \sum_{i, j =1}^d e_{i, t}^{(\varepsilon)} e_{j, t}^{(\varepsilon)} \partial_{i j}^2 V_p(Y_t^{(\varepsilon)}) dt \\
		&\quad+ \sqrt{\frac{2}{\beta}} \langle \nabla V_p(Y_t^{(\varepsilon)}), dW_t \rangle_{\R^d} 
		-2 \sqrt{\frac{2}{\beta}} h_\varepsilon( \| Z_t^{(\varepsilon)} \|_{\R^d}) \langle e_t^{(\varepsilon)}, \nabla V_p(Y_t^{(\varepsilon)}) \rangle_{\R^d} \langle e_t, dW_t \rangle_{\R^d}.
	\end{align*}	
	Since $V_p$ is convex for $p \geq 2$, its Hessian matrix is nonnegative-definite. 
	Thus, by Lemma \ref{p-th_Lyap}, we have
	\begin{align*}
		d V_p(Y_t^{(\varepsilon)}) 
		&\leq \{ C(p) - \lambda(p) V_p(Y_t^{(\varepsilon)}) \} dt 
		+ \sqrt{\frac{2}{\beta}} \langle \nabla V_p(Y_t^{(\varepsilon)}), dW_t \rangle_{\R^d} 
		-2 \sqrt{\frac{2}{\beta}} h_\varepsilon( \| Z_t^{(\varepsilon)} \|_{\R^d}) \langle e_t^{(\varepsilon)}, \nabla V_p(Y_t^{(\varepsilon)}) \rangle_{\R^d} \langle e_t^{(\varepsilon)}, dW_t \rangle_{\R^d}.
	\end{align*}
	Hence, since
	\[
	E[V_p(Y_t^{(\varepsilon)})] 
	\leq E[V_p(Y_0)] + C(p) t - \lambda(p) \int_0^t E[V_p(Y_s^{(\varepsilon)})] ds 
	\]
	holds, by $d ( e^{\lambda(p) t} E[V_p(Y_t^{(\varepsilon)})] ) \leq C(p) e^{\lambda(p) t} dt$, we obtain 
	\[
	e^{\lambda(p) t} E[V_p(Y_t^{(\varepsilon)})] 
	\leq E[V_p(Y_0)] + \frac{C(p)}{\lambda(p)} (e^{\lambda(p) t} -1).
	\]
	The bound for (\ref{LSDE_along_H_new}) can be proved in the same way. 
\end{proof}

The following lemma is an extension of Lemma 3.2 in \citep{Zhang} based on Young's inequality. 

\begin{lem}
	\label{Lp_bound_SGD}
	Assume that $F_k \in C^1(\R^d; \R)$ is $(m, b)$-dissipative and $M$-smooth for each $k$ and satisfies $\sup_{k \in \N} \| \nabla F_k(0) \|_{\R^d} \leq A$. 
	For fixed $\eta > 0$, we define the process $Y$ as
	\begin{align*}
		Y_t 
		= Y_{k \eta} - (t - k \eta) \nabla F_k(Y_{k \eta}) + \sqrt{2 / \beta} (W_t - W_{k \eta}),\qquad k \eta \leq t < (k+1) \eta.
	\end{align*}
	Then, for all $\ell \in \N$, there exists some $\eta_0 = O_{m, M, \ell}(1)$ and the following inequality holds uniformly on $0 < \eta \leq \eta_0$.
	\begin{align*}
		\sup_{t \geq 0} E[\| Y_t \|_{\R^d}^{2 \ell}]
		\leq O_{m, b, M, \beta, A, d, \ell} (1 + E[\| Y_0 \|_{\R^d}^{2 \ell}]).
	\end{align*}
\end{lem}

\begin{proof}
	Let $k \eta \leq t < (k+1) \eta$ and define $\Delta_{k, t} = Y_{k \eta} - (t - k \eta) \nabla F_k(Y_{k \eta})$ and $U_{k, t} = \sqrt{2 / \beta} (W_t - W_{k \eta})$. 
	If $s \in \N$ is odd, then we have $E[\langle \Delta_{k, t}, U_{k, t} \rangle_{\R^d}^s \,|\, Y_{k \eta}] = 0$. 
	Thus, there exist constants $a_j$ that depend only on $\ell \in \N$ such that
	\begin{align*}
		E[\| Y_t \|_{\R^d}^{2 \ell} \,|\, Y_{k \eta}] 
		&= E[(\|\Delta_{k, t}\|_{\R^d}^2 + 2 \langle \Delta_{k, t}, U_{k, t} \rangle_{\R^d} + \|U_{k, t}\|_{\R^d}^2)^{\ell} \,|\, Y_{k \eta}] \\
		&\leq \|\Delta_{k, t}\|_{\R^d}^{2 \ell} 
		+ \sum_{j=1}^{\ell} a_j E[\|\Delta_{k, t}\|_{\R^d}^{2 \ell - 2j} \|U_{k, t}\|_{\R^d}^{2 j} \,|\, Y_{k \eta}]. 
	\end{align*}
	By Young's inequality, for any $\varepsilon > 0$ and $1 \leq j \leq \ell -1$, 
	\begin{align*}
		a_j \|\Delta_{k, t}\|_{\R^d}^{2 \ell - 2j} \|U_{k, t}\|_{\R^d}^{2 j}
		= (\varepsilon \|\Delta_{k, t}\|_{\R^d}^{2 \ell - 2j}) (a_j \|U_{k, t}\|_{\R^d}^{2j} / \varepsilon) 
		\leq \frac{\ell - j}{\ell} \varepsilon^{\frac{\ell}{\ell - j}} \|\Delta_{k, t}\|_{\R^d}^{2 \ell} 
		+ \frac{j}{\ell} a_j^{\frac{\ell}{j}} \varepsilon^{- \frac{\ell}{j}} \|U_{k, t}\|_{\R^d}^{2 \ell}
	\end{align*}
	holds. 
	In particular, setting $\varepsilon = \{ \ell^{-1} m (t - k \eta) \}^{\frac{\ell - j}{\ell}}$, we obtain 
	\begin{align*}
		a_j \|\Delta_{k, t}\|_{\R^d}^{2 \ell - 2j} \|U_{k, t}\|_{\R^d}^{2 j}
		\leq \frac{m (t - k \eta)}{\ell} \|\Delta_{k, t}\|_{\R^d}^{2 \ell} 
		+ \frac{a_j^{\frac{\ell}{j}} \ell^{\frac{\ell - j}{j}}}{\{ m ( t - k \eta) \}^{\frac{\ell - j}{j}}} \|U_{k, t}\|_{\R^d}^{2 \ell}. 
	\end{align*}
	Hence, since $E[\|U_{k, t}\|_{\R^d}^{2 \ell}] = O_{\beta, d}((t - k \eta)^{\ell})$ and $\ell - \frac{\ell - j}{j} \geq 1$, for $\eta \leq 1$, there exists a constant $C_1$ independent of $k$ such that
	\begin{align*}
		E[\| Y_t \|_{\R^d}^{2 \ell} \,|\, Y_{k \eta}] 
		\leq \{ 1 +  m ( t - k \eta) \} \|\Delta_{k, t}\|_{\R^d}^{2 \ell} + (t - k \eta) C_1.
	\end{align*}
	
	For the first term, if $\eta \leq 1$, then there exist constants $b_j$ that depend only on $\ell \in \N$ such that
	\begin{align*}
		\|\Delta_{k, t}\|_{\R^d}^{2 \ell}
		&= (\|Y_{k \eta}\|_{\R^d}^2 
		- 2 (t - k \eta) \langle Y_{k \eta}, \nabla F_k(Y_{k \eta}) \rangle_{\R^d} 
		+ (t - k \eta)^2 \| \nabla F_k(Y_{k \eta}) \|_{\R^d}^2 )^{\ell} \\
		&\leq \|Y_{k \eta}\|_{\R^d}^{2 \ell} 
		- 2 \ell (t - k \eta) \|Y_{k \eta}\|_{\R^d}^{2 \ell - 2} \langle Y_{k \eta}, \nabla F_k(Y_{k \eta}) \rangle_{\R^d}
		+ (t- k \eta)^2 \sum_{j=1}^{2 \ell} b_j \|Y_{k \eta}\|_{\R^d}^{2 \ell - j} \| \nabla F_k(Y_{k \eta}) \|_{\R^d}^j. 
	\end{align*}
	Since the map $r \mapsto r^q$ is convex for $q \geq 2$, by $M$-smoothness of $F_k$ and the boundedness $\sup_{k \in \N} \|\nabla F_k(0) \|_{\R^d} \leq A$, we have
	\begin{align*}
		\| \nabla F_k(Y_{k \eta}) \|_{\R^d}^q
		\leq (M \| Y_{k \eta} \|_{\R^d} + A)^q
		\leq 2^{q-1} (M^q \| Y_{k \eta} \|_{\R^d}^q + A^q). 
	\end{align*}
	Thus, by the $(m, b)$-dissipativity of $F_k$, we can find constants $C_2$ and $C_3$ independent of $k$ so that
	\begin{align*}
		\|\Delta_{k, t}\|_{\R^d}^{2 \ell}
		&\leq \{ 1 - 2 \ell m (t - k \eta) + C_2 (t - k \eta)^2 \} \|Y_{k \eta}\|_{\R^d}^{2 \ell} 
		+ b (t - k \eta) \|Y_{k \eta}\|_{\R^d}^{2 \ell - 2} + C_3 (t - k \eta)^2. 
	\end{align*}
	In particular, when $\eta \leq 1$ is sufficiently small for $m$, $M$ and $\ell$, 
	\begin{align*}
		\|\Delta_{k, t}\|_{\R^d}^{2 \ell}
		&\leq \left\{ 1 - \frac{7}{4} \ell m (t - k \eta) \right\} \|Y_{k \eta}\|_{\R^d}^{2 \ell} 
		+ (t - k \eta) \{ b \|Y_{k \eta}\|_{\R^d}^{2 \ell - 2} + C_3 \}
	\end{align*}
	holds. 
	Thus, taking $K > 0$ sufficiently large so that $\frac{1}{4} \ell m K^{2 \ell} - b K^{2 \ell -2} + C_3 \geq 0$, if $\| Y_{k \eta} \|_{\R^d} \geq K$, then we have
	\begin{align}
		\label{leq_L_Case}
		\|\Delta_{k, t}\|_{\R^d}^{2 \ell}
		&\leq \left\{ 1 - \frac{3}{2} m (t - k \eta) \right\} \|Y_{k \eta}\|_{\R^d}^{2 \ell} 
		\leq \left\{ 1 - \frac{3}{2} m (t - k \eta) \right\} \|Y_{k \eta}\|_{\R^d}^{2 \ell} 
		+ (t - k \eta) \{ b K^{2 \ell - 2} + C_3 \}.
	\end{align}
	Since (\ref{leq_L_Case}) holds when $\| Y_{k \eta} \|_{\R^d} \leq K$, (\ref{leq_L_Case}) is always true. 
	Therefore, by
	\[
	\left(1 - \frac{3}{2} m (t - k \eta) \right) (1 + m (t - k \eta)) 
	= 1 - \frac{3}{2} m (t - k \eta) + m (t - k \eta) - \frac{3}{2} m^2 (t - k \eta)^2 
	\leq 1 - \frac{3}{2} m (t - k \eta),
	\]
	we can find a constant $C_4$ independent of $k$ such that
	\begin{align*}
		E[\|Y_t\|_{\R^d}^{2 \ell} \,|\, Y_{k \eta}] 
		&\leq \left(1 - \frac{1}{2} m (t - k \eta) \right) \|Y_{k \eta}\|_{\R^d}^{2 \ell} 
		+ (t - k \eta) C_4.
	\end{align*}
	From the aforementioned, we obtain 
	\begin{align*}
		E[\|Y_t\|_{\R^d}^{2 \ell}] 
		&\leq (1 - 2^{-1} m (t- k \eta)) E[\|Y_{k \eta}\|_{\R^d}^{2 \ell}] + (t- k \eta) C_4 \\
		&\leq (t- k \eta) C_4 + (1 - 2^{-1} m (t- k \eta)) \Big\{ (1 - 2^{-1} m \eta) E[\|Y_{(k-1) \eta}\|_{\R^d}^{2 \ell}] + \eta C_4 \Big\} \\
		&\leq \eta C_4 \{ 1 + (1 - 2^{-1} m (t- k \eta)) \} + (1 - 2^{-1} m (t- k \eta)) (1 - 2^{-1} m \eta) E[\|Y_{(k-1) \eta}\|_{\R^d}^{2 \ell}] \\
		&\leq \eta C_4 \{1 + (1- 2^{-1} m (t- k \eta)) \} + (1 - 2^{-1} m (t- k \eta)) (1 - 2^{-1} m \eta) \Big\{ (1 - 2^{-1} m \eta) E[\|Y_{(k-2) \eta}\|_{\R^d}^{2 \ell}] + \eta C_4 \Big\} \\
		&= \eta C_4 \{1 + (1- 2^{-1} m (t- k \eta)) + (1 - 2^{-1} m (t- k \eta)) (1 - 2^{-1} m \eta) \} + (1 - 2^{-1} m (t- k \eta)) (1 - 2^{-1} m \eta)^2 E[\|Y_{(k-3) \eta}\|_{\R^d}^{2 \ell}] \\
		&\leq \cdots \\
		&\leq \eta C_4 \left\{1 + (1- 2^{-1} m (t- k \eta)) \sum_{j=1}^{k-1} (1 - 2^{-1} m \eta)^j \right\} + (1 - 2^{-1} m (t- k \eta)) (1 - 2^{-1} m \eta)^k E[\|Y_0\|_{\R^d}^{2 \ell}] \\
		&\leq \eta C_4 \sum_{j=0}^{k-1} (1 - 2^{-1} m \eta)^j + (1 - 2^{-1} m (t- k \eta)) (1 - 2^{-1} m \eta)^k E[\|Y_0\|_{\R^d}^{2 \ell}] \\
		&\leq \eta C_4 \frac{1}{1 - (1 - 2^{-1} m \eta)} + (1 - 2^{-1} m (t- k \eta)) (1 - 2^{-1} m \eta)^k E[\|Y_0\|_{\R^d}^{2 \ell}] \\
		&= \frac{2C_4}{m} + (1 - 2^{-1} m (t- k \eta)) (1 - 2^{-1} m \eta)^k E[\|Y_0\|_{\R^d}^{2 \ell}].
	\end{align*}
	Thus, $E[\|Y_t\|_{\R^d}^{4 \ell}] \leq 2 C_4 / m + E[\|Y_0\|_{\R^d}^{2 \ell}]$ holds for all $t > 0$. 
\end{proof}


\begin{lem}
	\label{OneStepBound}
	Let $H : [0, \infty) \times \R^d \to \R$ and let $H(t, \cdot) \in C^1(\R^d; \R)$ be $M$-smooth for each $t \geq 0$. 
	For fixed $\eta > 0$, if we define $Y$ by 
	\begin{align}
		\label{LSDE_alongH}
		dY_t
		= - \nabla_x H(\lfloor t / \eta \rfloor \eta, Y_{\lfloor t / \eta \rfloor \eta}) dt + \sqrt{2 / \beta} dW_t, \quad t \geq 0,
	\end{align}
	then, for each $t \geq 0$, we have
	\[
	E[\| Y_t - Y_{\lfloor t / \eta \rfloor \eta} \|_{\R^d}^2] 
	\leq 
	\eta^2 E[( M \| Y_{\lfloor t / \eta \rfloor \eta} \|_{\R^d} + \| \nabla_x H(\lfloor t / \eta \rfloor \eta, 0) \|_{\R^d} )^2] 
	+ (2 d \eta) / \beta. 
	\]
\end{lem}


\begin{proof}
	According to (\ref{LSDE_alongH}), we have
	\[
	Y_t - Y_{\lfloor t / \eta \rfloor \eta}
	= - (t - \lfloor t / \eta \rfloor \eta) \nabla_x H(\lfloor t / \eta \rfloor \eta, Y_{\lfloor t / \eta \rfloor \eta}) dt + \sqrt{2 / \beta} (W_t - W_{\lfloor t / \eta \rfloor \eta})
	\]
	Thus, 
	\begin{align*}
		E[\| Y_t - Y_{\lfloor t / \eta \rfloor \eta} \|_{\R^d}^2] 
		&= \eta^2 E[\| \nabla_x H(\lfloor t / \eta \rfloor \eta, Y_{\lfloor t / \eta \rfloor \eta}) \|_{\R^d}^2] 
		+ (2 d \eta) / \beta \\
		&\leq \eta^2 E[( M \| Y_{\lfloor t / \eta \rfloor \eta} \|_{\R^d} + \| \nabla_x H(\lfloor t / \eta \rfloor \eta, 0) \|_{\R^d} )^2] 
		+ (2 d \eta) / \beta 
	\end{align*}
	holds. 
\end{proof}


The following lemma can be proved similarly to Lemma C.5 in \citep{Xu}. 


\begin{lem}
	\label{OneStepBatchBound}
	For each $k \in \N$, we have
	\[
	E[\| \nabla L_n(w) - \nabla L_{n, k}(w) \|_{\R^d}^2] 
	\leq \frac{4 (n - B)}{B (n - 1)} (M \| w \|_{\R^d} + A)^2,\qquad w \in \R^d.
	\]
\end{lem}

\subsection{Lemmas from \citep{Ebe}}

In this subsection, we use the notations introduced in subsection \ref{SUBSEC_Ebe_Notation}. 
The following lemma is in Section 2 in \citep{Ebe}. 


\begin{lem}
	\label{IneqDeriveFrom_S}
	Let $F, G \in C^1(\R^d; \R)$ be $(m, b)$-dissipative and $M$-smooth. 
	Under the notations of (\ref{DefOfS1}) and (\ref{DefOfS2}), if $(x, y) \notin S_1$, then
	\begin{align}
		\label{OutOfS1}
		\mathcal{L}_F \bar{V}_2(x) + \mathcal{L}_G \bar{V}_2(y)
		< 0
	\end{align}
	holds, where for $H \in C^1(\R^d; \R)$, $\mathcal{L}_H = - \langle \nabla H, \nabla \rangle_{\R^d} - \beta^{-1} \Delta$. 
	Furthermore, if $(x, y) \notin S_2$, then for any $\kappa > 0$, 
	\begin{align}
		\kappa \mathcal{L}_F \bar{V}_2(x) + \kappa \mathcal{L}_G \bar{V}_2(y)
		& \leq - \frac{\lambda}{2} \min \{ 1, 4C \kappa \} \{ 1 + \kappa \bar{V}_2(x) + \kappa \bar{V}_2(y) \} \label{OutOfS2}
	\end{align}
	holds. 
\end{lem}


The following condition (\ref{Ineq_kappa}) is the counterpart of (2.25) in \citep{Ebe}. 


\begin{lem}
	\label{Choose_kappa}
	Let
	\begin{align}
		\label{Condi_On_kappa}
		\kappa 
		\coloneqq \min \left\{ \frac{1}{2}, \frac{2}{C \beta (e^{2 R_1} - 1 - 2 R_1)} \exp \left\{ - \frac{M \beta}{8} R_1^2 \right\} \right\}
		\in (0, 1).
	\end{align}
	Then we have 
	\begin{align}
		\label{Ineq_kappa}
		\frac{1}{2 C \beta \kappa} 
		\geq \int_0^{R_1} \Phi(s) \varphi(s)^{-1} ds.
	\end{align}
\end{lem}


\begin{proof}
	Since $Q(\kappa) \in (0, 1]$, 
	\begin{align*}
		\int_0^{R_1} \Phi(s) \varphi(s)^{-1} ds
		&= \int_0^{R_1} \int_0^s \exp \left\{ \frac{M \beta}{8} (s^2 - r^2) + 2 Q(\kappa) (s - r)  \right\} dr ds \\
		&\leq \int_0^{R_1} \int_0^s \exp \left\{ \frac{M \beta}{8} (s^2 - r^2) + 2 (s - r) \right\} dr ds \\
		&\leq \exp \left\{ \frac{M \beta}{8} R_1^2 \right\} \int_0^{R_1} \int_0^s e^{ 2 (s - r) } dr ds \\
		&= \exp \left\{ \frac{M \beta}{8} R_1^2 \right\} \int_0^{R_1} \int_0^s e^{ 2 (s - r) } dr ds \\
		&= \frac{1}{2} \left\{ \frac{e^{2R_1}}{2} - \frac{1}{2} - R_1 \right\} \exp \left\{ \frac{M \beta}{8} R_1^2 \right\}
	\end{align*}
	holds. 
\end{proof}


The following three lemmas are in Section 5 in \citep{Ebe}. 


\begin{lem}
	\label{PropertyOfF}
	The function $f$ defined by (\ref{Def_Of_F}) is nonnegative and bounded on $[0, \infty)$ and satisfies $f(0) = 0$. 
	Furthermore, $f$ is continuous, increasing and concave on $\R$, and
	\[
	r \varphi(R_2) 
	\leq \Phi(r) 
	\leq 2 f(r) 
	\leq 2 \Phi(r) 
	\leq 2 r,\qquad 0 \leq r \leq R_2
	\]
	holds. 
\end{lem}


\begin{lem}
	\label{PropertyOfMu_f}
	$\mu_f((-\infty, 0] \cup (R_2, \infty)) = 0$ and $\mu_f(\{ R_i \}) \leq 0$ hold. 
\end{lem}


\begin{lem}
	\label{Esti_On_f}
	For any $ r \in (0, R_1) \cup (R_1, R_2)$, we have
	\begin{align*}
		f^{\prime \prime}(r)
		&\leq  -\left( \frac{M \beta}{4} r + 2 Q(\kappa) \right) f^\prime(r) - \frac{\zeta}{4} f(r) \chi_{(0, R_2)}(r) - \frac{\xi}{4} f(r) \chi_{(0, R_1)}(r).
	\end{align*}
\end{lem}

\subsection{Inequalities based on couplings}

\begin{lem}
	\label{BeforeWasserBound}
	(Lemma 6 in \citep{Ragi})
	Let $H \in C^1(\R^d; \R)$ satisfy $\| \nabla H(x) \|_{\R^d} \leq c_1 \| x \|_{\R^d} + c_2$. 
	Then for any probability measures $\mu$ and $\nu$ on $\R^d$ and any $\gamma \in \Pi(\mu, \nu)$, we have
	\[
	\left| \int_{\R^d} H(x) \mu(dx) - \int_{\R^d} H(y) \nu(dy) \right| 
	\leq \int_{\R^d \times \R^d} \left( \frac{c_1}{2} \| x \|_{\R^d} + \frac{c_1}{2} \| y \|_{\R^d} + c_2 \right) \| x - y \|_{\R^d} \gamma (dx dy).
	\] 
\end{lem}

\begin{lem}
	\label{ChangeToRho}
	Let $A > 0$. 
	For $\rho_2$ defined by (\ref{Rho2}), 
	\begin{align*}
		\left( \frac{M}{2} \| x \|_{\R^d} + \frac{M}{2} \| y \|_{\R^d} + A \right) \| x - y \|_{\R^d} 
		\leq 2 \exp \left( \frac{M \beta R_2^2}{8} + 2 R_2 \right) \max \left\{ 1, \frac{1}{R_2} \right\} \max \left\{ A + \frac{M}{2}, \frac{1}{\kappa} \left( \frac{A}{2} + M \right) \right\} \rho_2(x, y)
	\end{align*}
	holds. 
\end{lem}


\begin{proof}
	For $r \in [0, R_2)$, we have 
	\[
	f^\prime(r) 
	= \varphi(r) g(r) 
	\geq \frac{1}{2} \exp \left( -\frac{M \beta R_2^2}{8} - 2 R_2 \right).
	\]
	Thus, since $f(0) = 0$, if $\| x - y \|_{\R^d} \leq R_2$, then Taylor's theorem yields
	\begin{align*}
		f(\| x - y \|_{\R^d}) 
		\geq \frac{1}{2} \exp \left( -\frac{M \beta R_2^2}{8} - 2 R_2 \right) \| x - y \|_{\R^d}.
	\end{align*}
	Therefore, if $\| x - y \|_{\R^d} \leq R_2$, then by $r \leq 2^{-1}( 1 + r^2 )$, we obtain 
	\begin{align*}
		\left( \frac{M}{2} \| x \|_{\R^d} + \frac{M}{2} \| y \|_{\R^d} + A \right) \| x - y \|_{\R^d}  
		&\leq 2 \exp \left( \frac{M \beta R_2^2}{8} + 2 R_2 \right) \left( A + \frac{M}{2} + \frac{M}{4} \| x \|_{\R^d}^2 + \frac{M}{4} \| y \|_{\R^d}^2 \right) f(\| x - y \|_{\R^d}) \\
		&\leq 2 \exp \left( \frac{M \beta R_2^2}{8} + 2 R_2 \right) \max \left\{ A + \frac{M}{2}, \frac{M}{4 \kappa} \right\} \rho_2(x, y).
	\end{align*}
	On the other hand, if $\| x - y \|_{\R^d} > R_2$, then we have
	\[
	f( \| x - y \|_{\R^d}) 
	= f(R_2) 
	\geq \frac{R_2}{2} \exp \left( -\frac{M \beta R_2^2}{8} - 2 R_2 \right)
	\]
	and therefore
	\begin{align*}
		\left( \frac{M}{2} \| x \|_{\R^d} + \frac{M}{2} \| y \|_{\R^d} + A \right) \| x - y \|_{\R^d} 
		&\leq A + \left( \frac{A}{2} + M \right) (\| x \|_{\R^d}^2 + \| y \|_{\R^d}^2) \\
		&\leq \frac{2}{R_2} \exp \left( \frac{M \beta R_2^2}{8} + 2 R_2 \right) \max \left\{ A, \frac{1}{\kappa} \left( \frac{A}{2} + M \right) \right\} \rho_2(x, y) 
	\end{align*}
	holds. 
\end{proof}


The following lemma is a simple corollary to Lemmas \ref{BeforeWasserBound} and \ref{ChangeToRho}. 


\begin{lem}
	\label{Bound_By_Rho}
	Let $H \in C^1(\R^d; \R)$ be $M$-smooth and let $\mu$ and $\nu$ be probability measures on $\R^d$.
	Then we have
	\begin{align}
		| \mu(H) - \nu(H) | 
		\leq O_{m, b, M, \beta, \| \nabla H(0) \|_{\R^d}, d} (\mathcal{W}_{\rho_2}(\mu, \nu) ).
	\end{align}
\end{lem}


\bibliographystyle{plain}
\bibliography{article.bib}

\clearpage
	
\end{document}